\documentclass[12pt,a4paper]{article}

\usepackage{amsfonts}
\usepackage{amssymb,amsmath,amsthm}
\usepackage[utf8]{inputenc}
\usepackage[T1]{fontenc}
\usepackage{enumerate}
\usepackage{hyperref}
\usepackage{tikz}

\newtheorem{twr}{Theorem}[section]
\newtheorem{lem}[twr]{Lemma}
\newtheorem{exam}[twr]{Example}
\newtheorem{cor}[twr]{Corollary}
\theoremstyle{definition}
\newtheorem{remark}[twr]{Remark}

\DeclareMathOperator{\conv}{conv}
\DeclareMathOperator{\bd}{bd}
\DeclareMathOperator{\lin}{lin}
\DeclareMathOperator{\tr}{Tr}
\DeclareMathOperator{\inte}{int}

\newcommand\R{\mathbb{R}}
\newcommand\B{\mathbb{B}}
\newcommand\N{\mathbb{N}}
\newcommand\CC{\mathcal{C}}
\newcommand\CP{\mathcal{P}}
\newcommand\CU{\mathcal{U}}
\newcommand\CY{\mathcal{Y}}
\newcommand\CZ{\mathcal{Z}}
\newcommand\ms[1]{\mathcal{M}_s^{#1}}

\def\ds{1.5/\scale pt}

\newcommand\cemph[1]{\emph{\color{red}#1}}

\newcommand\extrafootertext[1]{%
    \bgroup
    \renewcommand\thefootnote{\fnsymbol{footnote}}%
    \renewcommand\thempfootnote{\fnsymbol{mpfootnote}}%
    \footnotetext[0]{#1}%
    \egroup
}
\providecommand{\keywords}[1]{%
\extrafootertext{\textit{Key words and phrases.} #1.}
}
\providecommand{\subjclass}[1]{%
\extrafootertext{2020 \textit{Mathematics Subject Classification.} #1.}
}

\begin{document}

\title{On Certain Extremal Banach-Mazur Distances and Ader's Characterization of Distance Ellipsoids}
\author{Florian Grundbacher \MakeLowercase{and} Tomasz Kobos}

\maketitle

\begin{abstract}
A classical consequence of the John Ellipsoid Theorem is the upper bound $\sqrt{n}$ on the Banach-Mazur distance between the Euclidean ball and any symmetric convex body in $\mathbb{R}^n$. Equality is attained for the parallelotope and the cross-polytope. While it is known that they are unique with this property for $n=2$ but not for $n \geq 4$, no proof of the characterization of the three-dimensional equality case seems to have ever been published. We fill this gap by showing that the parallelotope and the cross-polytope are the unique maximizers for $n=3$. Our proof is based on an extension of a characterization of distance ellipsoids due to Ader from $1938$, which predates the John Ellipsoid Theorem. Ader's characterization turns out to provide a decomposition similar to the John decomposition, which leads to a proof of the aforementioned $\sqrt{n}$ estimate that bypasses the concept of volumes and reveals precise information about the equality case. We highlight further consequences of Ader's characterization, including a proof of an unpublished result attributed to Maurey related to the uniqueness of distance ellipsoids. Additionally, we investigate more closely the role of the parallelogram as a maximizer in various problems related to the distance between planar symmetric convex bodies. We establish the stability of the parallelogram as the unique planar symmetric convex body with the maximal distance to the Euclidean disc with the best possible (linear) order. This uniqueness extends to the setting of pairs of planar $1$-symmetric convex bodies, where we show that the maximal possible distance between them is again $\sqrt{2}$, together with a characterization of the equality case involving the parallelogram.
\end{abstract}

\maketitle

\section{Introduction}

A \cemph{convex body} in $\R^n$ is a compact, convex subset with non-empty interior. It is called \cemph{(origin) symmetric} if it has a center of symmetry (at the origin). By $\| \cdot \|$ we shall understand the \cemph{Euclidean norm} on $\R^n$ with associated \cemph{Euclidean unit ball} $\B^n$ and \cemph{standard inner product} $\langle \cdot, \cdot \rangle$. The norm induced by an origin symmetric convex body $K \subseteq \R^n$ is denoted by $\|\cdot\|_K$. For two convex bodies $K, L \subseteq \R^n$, their \cemph{Banach-Mazur distance} is defined as 
\begin{equation}
\label{eq:dbm_def}
d_{BM}(K,L) = \inf \{\rho >0 : K+ u \subseteq T(L+v) \subseteq \rho (K+u) \},
\end{equation}
where the infimum is taken over all invertible linear operators $T : \R^n \to \R^n$ and vectors $u, v \in \R^n$. Here, $X + v = v + X = \{ x + v : x \in X \}$ and $\rho X = \{ \rho x : x \in X \}$ denote the \cemph{$v$-translation} and \cemph{$\rho$-dilatation} of $X \subseteq \R^n$ for $v \in \R^n$ and $\rho \in \R$, respectively. We write shortly $(-1) X = - X$.

The infimum in \eqref{eq:dbm_def} is in fact a minimum, and the vectors $u,v$ can be omitted if $K$ and $L$ are origin symmetric, i.e., the minimum occurs for $u = v = 0$. The Banach-Mazur distance can also be defined for any two normed spaces $X, Y$ of the same dimension as $\inf \|T\|_{X \to Y} \cdot \|T^{-1}\|_{Y \to X}$, where the infimum (actually a minimum) is taken over all invertible linear operators $T: X \to Y$ and the considered norms are the respective operator norms. The unit ball of an $n$-dimensional real normed space is an origin symmetric convex body in $\R^n$ and, vice versa, every origin symmetric convex body is the unit ball of exactly one norm on $\R^n$. It can be easily checked that the Banach-Mazur distance between two normed spaces is the distance between their unit balls, so these two approaches lead to the same notion. It is more convenient for us to use the language of symmetric convex bodies in the present paper, but all of our results could be immediately translated to the language of real normed spaces of a given dimension as well.

Perhaps the most well-known estimate of the Banach-Mazur distance follows from the famous John Ellipsoid Theorem. In $1948$, John published his seminal paper ``Extremum problems with inequalities as subsidiary conditions''~\cite{john}, where he characterized the unique maximal volume ellipsoid contained in a given convex body (see also \cite[Theorem~$15.3$]{tomczak}). We write $\bd(X)$ for the \cemph{boundary} of a set $X \subseteq \R^n$.

\begin{twr}[John Ellipsoid Theorem for symmetric convex bodies]
\label{thm:john}
Let $K \subseteq \R^n$ be an origin symmetric convex body such that $\B^n \subseteq K$. Then $\B^n$ is the unique ellipsoid of maximal volume contained in $K$ if and only if there exist contact points $u^1, \ldots, u^N \in \bd(K) \cap \bd(\B^n)$ and weights $\lambda_1, \ldots, \lambda_N>0$ such that for any $x \in \R^n$ we have
$$ x = \sum_{i=1}^{N} \lambda_i \langle x, u^i \rangle u^i.$$
In this case, $\sum_{i=1}^N \lambda_i = n$ and there exists a choice with $N \leq \frac{n(n+1)}{2}$.
\end{twr}

By a \cemph{John decomposition} we understand a decomposition of the form $x = \sum_{i=1}^{N} \lambda_i \langle x, u^i \rangle u^i$ as in the above theorem. An easy and well-known corollary of this result is that $K \subseteq \sqrt{n} \B^n$ if $K$ is an origin symmetric convex body for which the above conditions apply. Indeed, for any $x \in K$ we have
\begin{equation}
\label{eq:john_ineq}
\|x\|^2 = \langle x, x \rangle = \sum_{i=1}^N \lambda_i \langle x, u^i \rangle^2 \leq \sum_{i=1}^N \lambda_i = n,
\end{equation}
where we used that $| \langle x, u^i \rangle | \leq 1$ for all $i \in \{1, \ldots, N\}$, as $K$ and $\B^n$ share a common supporting hyperplane at the contact point $u^i \in \bd(K) \cap \bd(\B^n)$ that is perpendicular to $u^i$. In the language of the Banach-Mazur distance, this reads as $d_{BM}(K, \B^n) \leq \sqrt{n}$.
\newpage

This estimate plays a fundamental role in the study of Banach-Mazur distances and has profound implications in the local theory of Banach spaces. In particular, combined with a simple application of the triangle inequality, it implies that $d_{BM}(K, L) \leq n$ for any symmetric convex bodies $K, L \subseteq \R^n$. Much later, it turned out that this upper bound on the diameter of the symmetric Banach-Mazur compactum is of the right order, as a highly influential random construction of Gluskin~\cite{gluskin} provides random symmetric polytopes in $\R^n$ with a Banach-Mazur distance of at least $cn$, for some absolute constant $c>0$.

Given the fundamental role played by the estimate $d_{BM}(K, \B^n) \leq \sqrt{n}$, it comes as no surprise that characterizing the equality case of this inequality has gained considerable attention and that there are several papers on this topic (see for example \cite{tomczakstructure,lewis,milmanwolfson} and \cite[Chapter~$7$]{tomczak}). It is worth noting here that if $K \subseteq \R^n$ is a general convex body (not necessarily symmetric), then the John Ellipsoid Theorem for the general case easily implies that $d_{BM}(K, \B^n) \leq n$. It was proved by Leichtweiss already in $1959$~\cite{leichtweiss} (and rediscovered later by Palmon~\cite{palmon}) that equality holds in this estimate if and only if $K$ is a simplex. Moreover, the stability of the simplex as the unique convex body with the maximal Banach-Mazur distance to the Euclidean ball has been established in \cite{kobos}. Thus, the non-symmetric case has a clear equality condition, but the symmetric one turns out to be much less straightforward.

The characterization of the equality case in the estimate $d_{BM}(K, \B^n) \leq \sqrt{n}$ for a symmetric convex body $K \subseteq \R^n$ is complicated by the fact that if $n \geq 3$, there are at least two affinely non-equivalent symmetric convex bodies for which equality is achieved, namely $\CP^n = [-1,1]^n$, the $n$-dimensional \cemph{unit cube}, and its dual $\CC^n = (\CP^n)^\circ$, the $n$-dimensional \cemph{unit cross-polytope}. Here, we denote by $X^\circ = \{ y \in \R^n : \langle x, y \rangle \leq 1 \text{ for all } x \in X$\} the \cemph{polar} of a set $X \subseteq \R^n$. More generally, it is easy to see that if $K \subseteq \R^n$ is a symmetric convex body such that the maximal volume ellipsoid contained in $K$ (which we shortly call the \cemph{John ellipsoid} of $K$) and the minimal volume ellipsoid containing $K$ (which we refer to as the \cemph{Loewner ellipsoid} of $K$ for traditional reasons) are homothetic with ratio $\sqrt{n}$, then $d_{BM}(K, \B^n)=\sqrt{n}$. It turns out that in general, $\CP^n$ and $\CC^n$ are not the only symmetric convex bodies with the distance $\sqrt{n}$ to $\B^n$ and such constructions are widely known. For all dimensions $n \geq 4$, there exist infinitely many affinely non-equivalent symmetric convex bodies $K \subseteq \R^n$ such that $d_{BM}(K, \B^n)=\sqrt{n}$ (see Remark~\ref{rem:higher_dim} for details).

This leaves the cases of dimensions two and three, for which it is surprisingly hard to find any information in the literature. Let us provide some historical background to put our work in context. It turns out that the only maximizers in dimensions two and three are the obvious ones: it is only the parallelogram in dimension two, whereas there are only the parallelotope and the cross-polytope in dimension three. The former characterization can be traced back to independent works of John~\cite{johninert} from $1936$ and Behrend~\cite{behrend} from $1937$. Both proved that $d_{BM}(K, \B^2) \leq d_{BM}(\CP^2, \B^2) = \sqrt{2}$ for any symmetric convex body $K \subseteq \R^2$, with equality if and only if $K$ is a parallelogram. Neither used the language of Banach-Mazur distances and their proofs are somewhat convoluted from a modern point of view. Their works predate the John Ellipsoid Theorem, though Behrend essentially established the existence and uniqueness of the John ellipse in the symmetric planar case. Much later, in $1979$, Lewis~\cite{lewis} noted that the two-dimensional case is also a consequence of a more general result about Banach ideal norms (see Remark after Theorem~$2.2$ there). He attributed this observation to Figiel and Davis, seemingly unaware of the previous works of John and Behrend. The three discussed papers appear to be the only ones where an argument for the two-dimensional case is given.

To the best of our knowledge, no proof of the characterization of the three-dimensional equality case has ever been published. The only mention of the equality case in more recent literature appears to be the paper \cite{tomczakstructure} of Anisca, Tcaciuc, and Tomczak-Jaegermann from $2005$. They state that some experts in the field were aware of its characterization at the time, but no proof or reference is provided. They further point to a result reportedly proved by Maurey, which in particular implies that a symmetric convex body $K$ with $d_{BM}(K,\B^n) = \sqrt{n}$ in any dimension $n$ has a unique distance ellipsoid (see Theorem~\ref{thm:maurey} below for details). With such a result at hand, it would be quite simple to prove that the parallelotope and the cross-polytope are indeed the only symmetric convex bodies in $\R^3$ with the distance $\sqrt{3}$ to $\B^3$. We note that generally, the distance ellipsoid does not have to be unique (see \cite[Lemma~$2.2$]{praetorius}), but it is always unique in the planar case (see Corollary~\ref{cor:unique} below for details). The result of Maurey has been mentioned also in \cite[Remark~$1.2$]{arias} and \cite{praetorius}, but to the best of our best knowledge, no proof has been published until this day.

Surprisingly, the problem of characterizing the three-dimensional equality case was considered already in $1938$ in a largely forgotten paper by Ader \cite{ader}, which again predates the John Ellipsoid Theorem. Ader, who was a student of John in Kentucky, established the inequality $d_{BM}(K, \B^3) \leq \sqrt{3}$ for any three-dimensional symmetric convex body $K$. He noted that his method of proof would lead to a characterization of the cube and the cross-polytope as the only convex bodies for which equality occurs, but did not provide any further details. Ader's paper contains the remarkable condition~(ii) below, describing ellipsoids that realize the Banach-Mazur distance to a given symmetric convex body in terms of the contact points of the boundaries. He established this condition only in dimension three, seemingly relying on the dimension for geometric arguments in some parts of his proof, but his ideas are straightforward to extend to general dimension $n$. Moreover, it is immediate to see from a modern point of view that this characterization can be stated in the form of a decomposition on the contact points, akin to the John decomposition. The theorem below provides the extended result of Ader. We denote by $\ms{n}$ the $\frac{n(n+1)}{2}$-dimensional linear space of \cemph{real symmetric $n \times n$ matrices}.

\begin{twr}[Ader's characterization of distance ellipsoids]
\label{thm:ader_cond}
Let $K \subseteq \R^n$ be an origin symmetric convex body and let $R \geq r > 0$ be such that $r \B^n \subseteq K \subseteq R \B^n$. Then the following are equivalent:
\begin{enumerate}[(i)]
\item $d_{BM}(K,\B^n)=\frac{R}{r}$.
\item For every $A \in \ms{n}$, there exist an outer contact point $y \in \bd(K) \cap \bd(R \B^n)$ and an inner contact point $z \in \bd(K) \cap \bd(r \B^n)$ such that $\left \langle \frac{y}{R}, A \frac{y}{R} \rangle \geq \langle \frac{z}{r}, A \frac{z}{r} \right  \rangle$.
\item There exist integers $N, M \geq 1$, outer contact points $y^1, \ldots, y^N \in \bd(K) \cap \bd(R\B^n)$, inner contact points $z^1, \ldots, z^M \in \bd(K) \cap \bd(r\B^n)$, as well as weights $\lambda_1, \ldots, \lambda_N$, $\mu_1, \ldots, \mu_M > 0$, such that for any $x \in \R^n$ we have
$$ \sum_{i=1}^N \lambda_i \langle x, y^i \rangle y^i = \sum_{i=1}^M \mu_i \langle x, z^i \rangle z^i. $$
\end{enumerate}
In this case, $R^2 \sum_{i=1}^N \lambda_i = r^2 \sum_{i=1}^M \mu_i$ and there exists a choice with $N+M \leq \frac{n(n+1)}{2} + 1$.
\end{twr}

It is hard to believe that such a clear characterization of distance ellipsoids never gained any prominence and remained unnoticed for so many years. This is even more surprising when one considers that, according to the historical account \cite{kuhn}, John was actually more interested in the estimate $d_{BM}(K, \B^n) \leq \sqrt{n}$ for symmetric convex bodies $K \subseteq \R^n$ itself than in the ellipsoid of maximal volume and was strongly influenced by the earlier work of his student Ader. It should be emphasized that Ader worked exclusively with distance ellipsoids in his paper and never considered ellipsoids maximizing or minimizing the volume. As we demonstrate in Corollary~\ref{cor:estimate}, the decomposition form of Ader's condition also yields a proof of the $\sqrt{n}$~estimate. Although the alternative proof is less immediate than the classical one based on John decompositions \eqref{eq:john_ineq}, it can still be considered fairly simple. Interestingly, this shows that the optimal upper bound on the order of the diameter of the symmetric Banach-Mazur compactum can be obtained entirely without mention of the notion of volumes.

While the decomposition due to John became highly popular (his paper has been cited around $2000$ times), Ader's paper barely gets any mention in the published literature. It appears that, despite a substantial body of research focusing on the Banach-Mazur distance to the Euclidean ball, the condition of Ader was never recognized and it was also not rediscovered independently. The only closely related result we could find is the proof of \cite[Theorem~$2.1$]{lewis} in the previously mentioned paper of Lewis, which we discuss in detail in Remark~\ref{rem:general} below. To honor Ader's insights, which were truly ahead of their time, we call any decomposition as in Theorem~\ref{thm:ader_cond}~(iii) an \cemph{Ader decomposition}.

Our aim in the first part of this paper is to bring attention to Ader's forgotten result and to draw some consequences from it. We shall hereby focus largely on results concerning extreme situations for the distance to the Euclidean ball. The full proof of Theorem~\ref{thm:ader_cond} and its consequences are presented in Section~\ref{sec:ader}. This section also serves the purpose of clarifying and systematizing certain knowledge about the Banach-Mazur distance to the Euclidean ball, including proofs of some results that have never been published before. In particular, we present the aforementioned alternative proof of the inequality $d_{BM}(K, \B^n) \leq \sqrt{n}$ for symmetric convex bodies $K \subseteq \R^n$, which allows us to easily characterize the equality case in dimensions two and three (Theorem~\ref{thm:eqcase}). Moreover, we combine the Ader decomposition with additional insights about means of ellipsoids (Lemma~\ref{lem:mean_ellpsoid}) to obtain a proof of the unpublished result of Maurey related to the uniqueness of distance ellipsoids (Theorem~\ref{thm:maurey}).

In the second part of this paper, we focus more closely on the role of the parallelogram as a unique maximizer in several problems concerning the Banach-Mazur distance between certain classes of planar symmetric convex bodies. Most prominently and as discussed above, the parallelogram is the only planar symmetric convex body with distance $\sqrt{2}$ to the Euclidean disc. In Section~\ref{sec:2d}, we provide the following stability improvement of this characterization.

\begin{twr}
\label{thm:stability}
Let $K \subseteq \R^2$ be a symmetric convex body with $d_{BM}(K,\B^2) \geq \sqrt{2} - \varepsilon$ for some $\varepsilon > 0$. Then
$$ d_{BM}(K,\CP^2) < 1 + c \varepsilon, $$
where $c = 5 \sqrt{2} \approx 7.071$.
\end{twr}

Clearly, for any symmetric convex body $K \subseteq \R^2$ with $d_{BM}(K, \B^2) = \sqrt{2} - \varepsilon$, we have by the triangle inequality $d_{BM}(K, \CP^2) \geq \frac{\sqrt{2}}{\sqrt{2}-\varepsilon} \geq 1 + \frac{\varepsilon}{\sqrt{2}}$, so the linear order in the above estimate is optimal. As a corollary, we obtain that the symmetric Banach-Mazur compactum can be covered by two balls with radius $\frac{11 \sqrt{2}}{10 + \sqrt{2}} < 1.363$ and centers at $\B^2$ and $\CP^2$ (see Corollary~\ref{cor:BM_cover}).
\newpage

The distance problem for the Euclidean ball is not the only situation where the parallelogram is known to be part of the only maximizing pair of convex bodies. A well-known result by Stromquist~\cite{stromquist} states that the distance between any two planar symmetric convex bodies is at most $\frac{3}{2}$, where equality occurs precisely for the parallelogram and the affine-regular hexagon. In Section~\ref{sec:1sym}, we move to another setting, which can be seen as an intermediate step between the two aforementioned problems, namely the problem of estimating the distance between planar $1$-symmetric convex bodies (defined below). This means that instead of bounding the distance to the Euclidean disc, we estimate the distance between pairs of arbitrary $1$-symmetric convex bodies (also restricting the second convex body since the distance between a $1$-symmetric convex body and an arbitrary symmetric one can still be maximal possible).

A convex body $K \subseteq \R^n$ is called \cemph{$1$-symmetric} if for any point $x=(x_1, x_2, \ldots, x_n) \in K$, we have that $(\sigma_1 x_{\pi(1)}, \sigma_2 x_{\pi(2)}, \ldots, \sigma_n x_{\pi(n)}) \in K$ for any choice of signs $\sigma \in \{-1, 1\}^n$ and any permutation $\pi: \{1, 2, \ldots, n\} \to \{1, 2, \ldots, n\}$. $1$-symmetric convex bodies (or $1$-symmetric normed spaces) include the unit balls of the $\ell^n_p$-norms and have been studied extensively, as Banach spaces with $1$-symmetric bases are a topic of great interest even in infinite dimensions (see, e.g., \cite{tomczak}). Tomczak-Jaegermann proved in \cite{tomczaksymm} that the maximal distance between two $1$-symmetric convex bodies in $\R^n$ is of much smaller order than the maximal possible distance between two arbitrary symmetric convex bodies: It is not greater than $C \sqrt{n}$, where $C = \frac{2^{25/2}}{(\sqrt[4]{2}-1)^2}$. However, with the constant $C$ being quite large, this estimate does not imply anything in the case of small dimensions, which thus require a more detailed analysis. We study the maximal distance between two $1$-symmetric convex bodies in dimension two and prove that in this case, it actually coincides with the maximal distance between the Euclidean disc and planar symmetric convex bodies. Interestingly, the equality condition is not a simple one, and it again involves $\CP^2$ as indicated above.

\begin{twr}
\label{thm:1sym}
Let $K, L \subseteq \R^2$ be $1$-symmetric convex bodies. Then
$$d_{BM}(K, L) \leq \sqrt{2}.$$
Moreover, equality holds if and only if one of the convex bodies $K, L$ is a square (let it be $L$) and the second one satisfies for every $x \in \R^2$ the condition
\begin{equation}
\label{eq:thm_1sym_cond}
\|x \|_K \| x \|_{\varphi(K^\circ)} \geq \| x\|^2,
\end{equation}
where $\varphi \colon \R^2 \to \R^2$ is a rotation by $45^\circ$.
\end{twr}

Let us point out that the rotation direction of $\varphi$ above does not matter, as planar $1$-symmetric convex bodies are invariant under rotation by $90^\circ$. Therefore, rotating $K^\circ$ by $45^\circ$ in either direction yields the same result. We further note that the condition \eqref{eq:thm_1sym_cond} is a rather unusual one for characterizing the equality of some Banach-Mazur distance. It should be compared with the fact that for an arbitrary origin symmetric convex body $K \subseteq \R^n$, one has $\|x \|_K \| x \|_{K^\circ} \geq \| x\|^2$ for any $x \in \R^n$. In particular, if $K$ is a $1$-symmetric convex body such that $\varphi(K)=K$ (i.e., $K$ is invariant under rotation by $45^\circ$), then $d_{BM}(K, \CP^2)=\sqrt{2}$. This generalizes a result of Lassak in \cite{lassak}, where it is proved that every regular $8j$-gon is at distance $\sqrt{2}$ to $\CP^2$. However, there are also examples of $1$-symmetric convex bodies satisfying this condition \eqref{eq:thm_1sym_cond} that are not invariant under rotation by $45^\circ$ (see Example~\ref{ex:stab_equal_cond_bodies}).

Throughout the paper, we write $\inte(X)$, $\conv(X)$, $\lin(X)$ for the \cemph{interior}, \cemph{convex hull} and \cemph{linear span} of $X \subseteq \R^n$, respectively. The \cemph{segment} connecting two points $v, w \in \R^n$ is denoted by $[v, w]$. If $K \subseteq \R^n$ is a convex body, then we say that a hyperplane $H \subseteq \R^n$ \cemph{supports} $K$ at $x \in K$ if $x \in H$ and $H$ does not intersect the interior of $K$. We write $U^\perp$ for the \cemph{orthogonal complement} of a linear subspace $U \subseteq \R^n$.

\section{Ader's characterization and its consequences}
\label{sec:ader}

We begin this section with the proof of Theorem~\ref{thm:ader_cond}. The proof of the implication from (i) to (ii) very closely follows Ader's original reasoning in \cite{ader}. The reverse implication is also based on his idea, which we present in a more abstract form to make it easier to generalize to the $n$-dimensional case. The decomposition form (iii) was not stated in his paper. This might be a reason why Ader did not prove the inequality $d_{BM}(K,\B^n) \leq \sqrt{n}$ for general dimensions. To prove this inequality, the decomposition form seems to be more convenient to work with than the non-separation condition, i.e., condition (ii). We denote by $I_n \in \ms{n}$ the \cemph{$n \times n$ identity matrix}. For matrices $A,B \in \ms{n}$, we write $\|A\|$ for the usual operator norm when $A$ is considered as an operator from $\R^n$ to $\R^n$ equipped with the Euclidean norm, and, with $\tr$ denoting the \cemph{trace}, further $\langle A, B \rangle_F = \tr(A B)$ for their \cemph{Frobenius inner product}. Note for $x \in \R^n$ that
\begin{equation}
\label{eq:frobenius}
\langle x, A x \rangle = \tr(x^T A x) = \tr(A x x^T) = \langle A, x x^T \rangle_F.
\end{equation}

\begin{proof}[Proof of Theorem~\ref{thm:ader_cond}]
We begin with the implication from (i) to (ii). Let us fix some matrix $A \in \ms{n} \setminus \{0\}$. For $\varepsilon \in \left (0,\frac{1}{\|A\|} \right )$, we define a linear operator $T_\varepsilon \colon \R^n \to \R^n$ by the relation $T_\varepsilon(x) := x + \varepsilon A x = (I_n + \varepsilon A) x$. Since $\|\varepsilon A\| < 1$, the Neumann series show that $T_\varepsilon$ is invertible. Thus, there exist $y^\varepsilon, z^\varepsilon \in \bd(K)$ such that $\|T_\varepsilon(y^\varepsilon)\| = \max \{ \|y'\| : \ y' \in \bd(T_\varepsilon(K)) \}$ and $\|T_\varepsilon(z^\varepsilon)\| = \min \{ \|z'\| : \ z' \in \bd(T_\varepsilon(K)) \} > 0$. Since $y^\varepsilon, z^\varepsilon \in \bd(K)$, we clearly have $\|y^\varepsilon\| \leq R$ and $\|z^\varepsilon\| \geq r$. The assumption $d_{BM}(K,\B^n) = \frac{R}{r}$ now implies that
\begin{equation}
\label{eq:Ader_cond_ineq}
1 \leq \frac{r}{R} \cdot \frac{\max \{ \|y'\| : \ y' \in \bd(T_\varepsilon(K)) \}}{\min \{ \|z'\| : \ z' \in \bd(T_\varepsilon(K))\}}
= \frac{r \|y^\varepsilon\|}{R \|z^\varepsilon\|} \cdot \frac{\left\| (I_n + \varepsilon A) \frac{y^\varepsilon}{\|y^\varepsilon\|}\right\|}{\left\| (I_n + \varepsilon A) \frac{z^\varepsilon}{\|z^\varepsilon\|}\right\|} \leq \frac{\left\| (I_n + \varepsilon A) \frac{y^\varepsilon}{\|y^\varepsilon\|}\right\|}{\left\| (I_n + \varepsilon A) \frac{z^\varepsilon}{\|z^\varepsilon\|}\right\|}.
\end{equation}
For any $x \in \R^n \setminus \{0\}$, we have that
$$ \left\| (I_n + \varepsilon A) \frac{x}{\|x\|} \right\|^2 = 1 + 2 \varepsilon \left\langle \frac{x}{\|x\|}, A \frac{x}{\|x\|} \right\rangle + \varepsilon^2 \left\langle A \frac{x}{\|x\|}, A \frac{x}{\|x\|} \right\rangle $$
and for any $u,v \in \bd(\B^n)$ that $|\langle Au, Au \rangle - \langle Av, Av \rangle| \leq 2 \|A\|^2$. Hence, we obtain from \eqref{eq:Ader_cond_ineq} that
$$ \left\langle \frac{y^\varepsilon}{\|y^\varepsilon\|}, A \frac{y^\varepsilon}{\|y^\varepsilon\|} \right\rangle \geq \left\langle \frac{z^\varepsilon}{\|z^\varepsilon\|}, A \frac{z^\varepsilon}{\|z^\varepsilon\|} \right\rangle - \varepsilon \|A\|^2. $$
\newpage

By compactness of $\bd(K)$, we can choose a sequence $\{ \varepsilon_m \}_{m \geq 1}$ in $\left( 0,\frac{1}{\|A\|} \right)$ with $\varepsilon_m \to 0$ for $m \to \infty$ such that $y^{\varepsilon_m} \to y$ and $z^{\varepsilon_m} \to z$ for some $y,z \in \bd(K)$. Now,
\begin{align*}
\left\langle \frac{y}{\|y\|}, A \frac{y}{\|y\|} \right\rangle
& = \lim_{m \to \infty} \left\langle \frac{y^{\varepsilon_m}}{\|y^{\varepsilon_m}\|}, A \frac{y^{\varepsilon_m}}{\|y^{\varepsilon_m}\|} \right\rangle
\\
& \geq \lim_{m \to \infty} \left\langle \frac{z^{\varepsilon_m}}{\|z^{\varepsilon_m}\|}, A \frac{z^{\varepsilon_m}}{\|z^{\varepsilon_m}\|} \right\rangle - \varepsilon_m \|A\|^2
= \left\langle \frac{z}{\|z\|}, A \frac{z}{\|z\|} \right\rangle.
\end{align*}
Finally, we have $\|y\| = \lim_{m \to \infty} \|y^{\varepsilon_m}\| = \lim_{m \to \infty} \max \{ \|T_{\varepsilon_m}(y')\| : y' \in \bd(K) \} = R$ and similarly $\|z\| = r$. Altogether, $y$ and $z$ are possible choices of the required contact points.

Next, we prove that (ii) implies (i) by contrapositive. Suppose that $\frac{R}{r} > d :=d_{BM}(K, \B^n)$. There exists an origin symmetric ellipsoid $E \subseteq \R^n$ such that $E \subseteq K \subseteq d E$. Let $S \in \ms{n}$ be a matrix with $\|x\|_E^2 = \langle x, S x \rangle$ for all $x \in \R^n$ and define the matrix $A := S - \frac{1}{r^2} I_n \in \ms{n}$. For $y \in \bd(K) \cap \bd(R\B^n)$, it follows from $y \in K \subseteq d E$ that $$ \langle y, A y \rangle = \langle y, S y \rangle - \frac{\langle y, y \rangle}{r^2} = \|y\|_E^2 - \frac{\|y\|^2}{r^2} \leq d^2 - \frac{R^2}{r^2} < 0. $$ For $z \in \bd(K) \cap \bd(r \B^n)$, we further have $\|z\|_E \geq 1$ since $z \in \bd(K)$ and $E \subseteq K$. Thus,
$$ \langle z, A z \rangle = \langle z, S z \rangle - \frac{\langle z, z \rangle}{r^2} = \|z\|_E^2 - \frac{\|z\|^2}{r^2} \geq 1 - \frac{r^2}{r^2}=0.$$
Clearly, the same inequalities are true when $y$ is replaced with $\frac{y}{R}$ and $z$ is replaced with $\frac{z}{r}$. Therefore, the matrix $A$ violates (ii) as desired.

We move to the implication from (ii) to (iii). Let us define two sets of rank-one matrices by $\CY := \{ \tilde{y} \tilde{y}^T : R \tilde{y} \in \bd(K) \cap \bd(R\B^n) \} \subseteq \ms{n}$ and $\CZ := \{ \tilde{z} \tilde{z}^T : r \tilde{z} \in \bd(K) \cap \bd(r\B^n) \} \subseteq \ms{n}$. Both of them are compact by the compactness of $\bd(K) \cap \bd(R\B^n)$ and $\bd(K) \cap \bd(r\B^n)$. By the assumption (ii) and \eqref{eq:frobenius}, there exists no $A \in \ms{n}$ such that all $y \in \bd(K) \cap \bd(R\B^n)$ and $z \in \bd(K) \cap \bd(r\B^n)$ satisfy
$$ \left\langle A, \frac{y y^T}{R^2} \right\rangle_F = \left\langle \frac{y}{R}, A \frac{y}{R} \right\rangle < \left\langle \frac{z}{r}, A \frac{z}{r} \right\rangle = \left\langle A, \frac{z z^T}{r^2} \right\rangle_F. $$
Consequently, the compact sets $\conv(\CY)$ and $\conv(\CZ)$ cannot be strongly separated. They thus intersect, so there exist integers $N,M \geq 1$, points $\tilde{y}^1, \ldots, \tilde{y}^N \in \bd(\frac{K}{R}) \cap \bd(\B^n)$, $\tilde{z}^1, \ldots, \tilde{z}^M \in \bd(\frac{K}{r}) \cap \bd(\B^n)$, and weights $\tilde{\lambda}_1, \ldots, \tilde{\lambda}_N, \tilde{\mu}_1, \ldots, \tilde{\mu}_M > 0$ with $\sum_{i=1}^N \tilde{\lambda}_i = \sum_{i=1}^M \tilde{\mu}_i = 1$ such that
\begin{equation}
\label{eq:decomp}
\sum_{i=1}^N \tilde{\lambda}_i \tilde{y}^i (\tilde{y}^i)^T = \sum_{i=1}^M \tilde{\mu}_i \tilde{z}^i (\tilde{z}^i)^T.
\end{equation}
Equivalently, we have for all $x \in \R^n$ that
$$ \sum_{i=1}^N \tilde{\lambda}_i \langle x, \tilde{y}^i \rangle \tilde{y}^i = \sum_{i=1}^M \tilde{\mu}_i \langle x, \tilde{z}^i \rangle \tilde{z}^i. $$
For the desired decomposition, we simply define $y^i := R \tilde{y}^i$, $z^i := r \tilde{z}^i$, $\lambda_i = \frac{\tilde{\lambda}_i}{R^2}$, and $\mu_i := \frac{\tilde{\mu}_i}{r^2}$.
\newpage

For the implication from (iii) to (ii), we first note that comparing traces of both sides of the decomposition in (iii) yields
$$ \omega := R^2 \sum_{i=1}^N \lambda_i = r^2 \sum_{i=1}^M \mu_i > 0. $$
Therefore, it is straightforward to deduce (ii) from (iii) by simply reversing all of the steps above. The only additional observation needed once one arrives at the equality \eqref{eq:decomp} is that all weights need to be rescaled by $\frac{1}{\omega}$ to obtain a point in $\conv(\CY) \cap \conv(\CZ)$.

Lastly, we get the estimate $N+M \leq \frac{n (n+1)}{2} + 1$ from Kirchberger's theorem \cite[Theorem~$1.3.11$]{schneider}: It states that if two compact sets in $\R^k$ cannot be strongly separated, then one can choose two subsets of these sets with at most $k+2$ points in total such that already these subsets cannot be strongly separated. In our case, this would give us $N+M \leq \frac{n (n+1)}{2} + 2$ points that satisfy \eqref{eq:decomp} since $\dim(\ms{n}) = \frac{n (n+1)}{2}$. To further reduce the upper bound on $N+M$ by $1$, we note that all matrices in $\CY$ and $\CZ$ have trace $1$, i.e., they all live in the affine subspace $\CU := \{ A \in \ms{n} : \langle I_n, A \rangle_F = \tr(A) = 1 \}$. Since $\CY$ and $\CZ$ intersect, they also cannot be strongly separated within $\CU$. Applying Kirchberger's Theorem relative to $\CU$ gives the claimed bound on $N+M$ for an appropriate choice.
\end{proof}

\begin{remark}
\label{rem:general}
In the process of proving \cite[Theorem~$2.1$]{lewis} (more details on the result itself in Remark~\ref{rem:contact_pairs} below), Lewis established conditions that are close to the necessity of the non-separation condition (ii) and decomposition condition (iii) in the general case. For origin symmetric convex bodies $K,L \subseteq \R^n$ with $r L \subseteq K \subseteq R L$ in optimal Banach-Mazur distance position, the non--separation condition should read as follows: for every real $n \times n$~matrix $A$, there exists an outer contact pair $(y, u)$ (i.e., $y \in \bd(K) \cap \bd(RL)$ and $u \in \bd(K^{\circ}) \cap \bd((RL)^{\circ})$ with $\langle y, u \rangle = 1$) and an inner contact pair $(z, v)$ (which means $z \in \bd(K) \cap \bd(rL)$ and $v \in \bd(K^{\circ}) \cap \bd((rL)^{\circ})$ with $\langle z, v \rangle = 1$), such that $\langle y, Au \rangle \geq \langle z, Av \rangle$. The corresponding decomposition condition states the existence of outer contact pairs $(y^1,u^1), \ldots, (y^N,u^N)$, inner contact pairs $(z^1,v^1), \ldots, (z^M,u^M)$, and weights $\lambda_1, \ldots, \lambda_N, \mu_1, \ldots, \mu_M > 0$ such that we have for all $x \in \R^n$ the equality
$$ \sum_{i=1}^N \lambda_i \langle x, u^i \rangle y^i = \sum_{i=1}^M \mu_i \langle x, v^i \rangle z^i. $$
Lewis' proof of \cite[Theorem~$2.1$]{lewis} provides slightly weaker conditions. The only difference is that $v \in \bd(K^{\circ})$ (resp.~$v^i \in \bd(K^\circ)$) may fail, where it is not apparent how to modify the proof to ensure these conditions in general as well. It is possible if $L$ is smooth, however: Clearly, the convex bodies $K$ and $rL$ share a supporting hyperplane at a contact point $z \in \bd(K) \cap \bd(rL)$. If $L$ is smooth, then this hyperplane must be the unique hyperplane supporting $r L$ at $z$, i.e., any hyperplane that supports $rL$ at $z$ also supports $K$. By changing the roles of $K$ and $L$ or using a polarity argument, we see that it would be enough if one of $K$ or $L$ is smooth or strictly convex to obtain the above necessary conditions.
\newpage

Let us already point out that the above necessary conditions would, in contrast to the Euclidean case, not be enough to guarantee the optimal Banach-Mazur distance position in general. This can be seen by the example of $\CC^n \subseteq \CP^n \subseteq n\CC^n$, where appropriately chosen inner and rescaled outer contact pairs form respective John decompositions. However, the Banach-Mazur distance between $\CP^n$ and $\CC^n$ is of order $\sqrt{n}$ (see for example \cite[Proposition~$37.6$]{tomczak}), so in general much smaller than $n$. It is not difficult to slightly modify this counterexample to additionally make both of the convex bodies smooth and strictly convex. Extending at least the necessary conditions to general pairs of convex bodies (i.e., without any regularity or even symmetry assumptions) would be an interesting direction for future research.
\end{remark}

In the following corollary, we show how to obtain the upper bound $d_{BM}(K, \B^n) \leq \sqrt{n}$ for symmetric convex bodies $K \subseteq \R^n$ using the Ader decomposition.

\begin{cor}
\label{cor:estimate}
For every symmetric convex body $K \subseteq \R^n$ we have $d_{BM}(K, \B^n) \leq \sqrt{n}$.
\end{cor}
\begin{proof}
Let $K$ be origin symmetric and let $R \geq r > 0$ be such that $r \B^n \subseteq K \subseteq R \B^n$ and $\frac{R}{r} = d_{BM}(K,\B^n)$. For an Ader decomposition as in Theorem~\ref{thm:ader_cond}~(iii), take $A \in \ms{n}$ such that for all $x \in \R^n$ we have
$$ A x = \sum_{i=1}^N \lambda_i \langle x, y^i \rangle y^i = \sum_{i=1}^M \mu_i \langle x, z^i \rangle z^i. $$
Then
$$ R^2 \sum_{i=1}^N \lambda_i = \tr(A) = r^2 \sum_{i=1}^M \mu_i > 0, $$
so after rescaling all $\lambda_i$ and $\mu_j$ by the same positive factor if necessary, we may assume $\tr(A) = 1$ and consequently $\sum_{i=1}^N \lambda_i = \frac{1}{R^2}$, $\sum_{i=1}^M \mu_i = \frac{1}{r^2}$. Next, we observe that since $z^i$ is a common boundary point of $K$ and $r \B^n$, we have $| \langle x, z^i \rangle | \leq r^2$ for every $x \in K$. Therefore,
\begin{align*}
\sum_{i=1}^M \mu_i \langle A z^i, z^i \rangle
& = \sum_{i=1}^M \mu_i \left \langle \sum_{j=1}^{N} \lambda_j \langle z^i, y^j \rangle y^j, z^i \right \rangle
= \sum_{i, j} \mu_i \lambda_j \langle z^i, y^j \rangle^2
\leq r^4 \sum_{i, j} \mu_i \lambda_j
\\
& = r^4 \left ( \sum_{i=1}^{M} \mu_i \right ) \left ( \sum_{j=1}^{N} \lambda_j \right ) = \frac{r^2}{R^2}
= \frac{1}{d_{BM}(K, \B^n)^2}.
\end{align*}
On the other hand, we obtain from \eqref{eq:frobenius} that
$$ \sum_{i=1}^M \mu_i \langle A z^i, z^i \rangle = \sum_{i=1}^M \mu_i \left\langle A, z^i (z^i)^T \right\rangle_F = \left\langle A, \sum_{i=1}^M \mu_i z^i (z^i)^T \right\rangle_F = \langle A, A \rangle_F. $$
Applying the Cauchy-Schwarz inequality for the Frobenius inner product yields
$$ n \langle A, A \rangle_F = \langle I_n, I_n \rangle_F \cdot \langle A, A \rangle_F \geq \langle I_n, A \rangle_F^2 = \tr(A)^2 = 1. $$
This shows that $\langle A, A \rangle_F \geq \frac{1}{n}$ and hence $d_{BM}(K, \B^n) \leq \sqrt{n}$.

By the equality case in the Cauchy-Schwarz inequality, the equality $d_{BM}(K,\B^n) = \sqrt{n}$ holds if and only if $A$ is a positive multiple of $I_n$ and $|\langle z^i, y^j \rangle| = r^2$ for all $i$ and $j$.
\end{proof}

The fact that $A$ is a positive multiple of the identity matrix in the equality case means that there exist John decompositions supported on the (rescaled) inner and outer contact points, respectively. This allows us to easily characterize the equality case in dimensions two and three. Before that, we need one additional simple observation.

\begin{lem}
\label{lem:n_John_decomp}
Suppose that vectors $u^1, \ldots, u^n \in \bd(\B^n)$ and weights $\lambda_1, \ldots, \lambda_n > 0$ form a John decomposition. Then the set $\{ u^1, \ldots, u^n \}$ is an orthonormal basis of $\R^n$.
\end{lem}
\begin{proof}
We only need to show that the vectors $u^i$ are pairwise orthogonal. From the John decomposition, we have for all $j \in \{1, \ldots, n\}$
$$1 = \langle u^j, u^j \rangle = \sum_{i=1}^n \lambda_i \langle u^j, u^i \rangle^2 = \lambda_j + \sum_{\substack{i=1 \\ i \neq j}}^n \lambda_i \langle u^j, u^i \rangle^2.$$
Summing over all $j$ yields
$$n = \sum_{j=1}^n \lambda_j + \sum_{j=1}^n \sum_{\substack{i=1 \\ i \neq j}}^n \lambda_i \langle u^j, u^i \rangle^2 = n + \sum_{j=1}^n \sum_{\substack{i=1 \\ i \neq j}}^n \lambda_i \langle u^j, u^i \rangle^2.$$
Every remaining term of the form $\lambda_i \langle u^j, u^i \rangle^2$ is non-negative, so all of them must be zero. Since $\lambda_i > 0$, this implies $\langle u^j, u^i \rangle = 0$ for all $i,j \in \{1, \ldots, n\}$ such that $i \neq j$, so the conclusion follows.
\end{proof}

\begin{twr}
\label{thm:eqcase}
Let $K \subseteq \R^n$ be a symmetric convex body such that $d_{BM}(K, \B^n)=\sqrt{n}$. If $n \leq 3$, then $K$ is an affine transformation of $\CP^n$ or $\CC^n$.
\end{twr}
\begin{proof}
Let us suppose that $\B^n \subseteq K \subseteq \sqrt{n}\B^n$. The proof of the inequality $d_{BM}(K,\B^n) \leq \sqrt{n}$ based on the Ader decomposition (Corollary~\ref{cor:estimate}) shows that any Ader decomposition for $K$ must be comprised of two John decompositions, which are supported on the inner and rescaled outer contact points, respectively. Additionally, we know from Theorem~\ref{thm:ader_cond} that these John decompositions can be chosen to be supported on at most $\frac{n(n+1)}{2} + 1$ points in total. It is also clear that any John decomposition must be supported on at least $n$ points (not including symmetric pairs). If one of them is supported on precisely $n$ points, then those points must form an orthogonal set by Lemma~\ref{lem:n_John_decomp}. Based on these observations, we shall consider the two cases for $n = 2$ and $n = 3$ separately.

If $n=2$, then both John decompositions are supported on precisely two points since $\frac{n(n+1)}{2} + 1 = 4 = n + n$. Let $\pm x, \pm y \in \bd(K) \cap \bd(\B^2)$ be the corresponding inner contact points (where $x \neq \pm y$). In this case, $x, y$ form an orthonormal basis and we have $K \subseteq P$ for the square $P$ induced by the tangents to $\B^2$ at the points $\pm x, \pm y$. Moreover, $K$ must contain all four vertices of $P$ since they are the only points in $P \cap \bd(\sqrt{2} \B^2)$, i.e., the only possible pairs of outer contact points for $K$ and $\sqrt{2}\B^2$. It follows that $K=P$ is a square.

If $n=3$, then at least one of the decompositions is supported on precisely three points, while the other decomposition is supported on three or four points. We first assume that the decomposition for $\B^3$ is supported on precisely three points $x,y,z$. We know that $x, y, z$ form an orthonormal basis, so let $C$ be the cube induced by the planes supporting $\B^3$ at $\pm x, \pm y, \pm z$. Then $C$ intersects $\sqrt{3} \B^3 $ in precisely its eight vertices, which are therefore the only possible outer contact points. Since none of these vertices are orthogonal to each other, it is not possible for the John decomposition for $K$ and $\sqrt{3}\B^3$ to be supported on only three points. Hence, the decomposition must be supported on four points, so that $K$ must contain all vertices of $C$. We conclude that $K = C$ is a cube. If instead the decomposition for $\sqrt{3}\B^3$ is based on precisely three points, then we can use a duality argument. Indeed, we still have $d_{BM}(K^\circ, \B^3) = \sqrt{3}$ and the previous reasoning applied to $K^\circ$ shows that $K^\circ$ is a cube. Consequently, $K$ is a cross-polytope and the proof is complete.
\end{proof}

The above arguments no longer work if $n \geq 4$ since there are simply too many possibilities to split the number of contact points between the decompositions for $\B^n$ and $\sqrt{n}\B^n$ in this case. In the remark below, we provide an explicit example of symmetric convex bodies that are affinely non-equivalent to $\CP^n$ and $\CC^n$ with the distance $\sqrt{n}$ to $\B^n$ for all $n \geq 4$. Essentially the same example can be found in a paper of Leichtweiss \cite{leichtweiss}, which is, however, available only in German. For the convenience of non-German speaking readers, we provide the example here as well (with a simplified proof based on Theorem~\ref{thm:ader_cond}).

\begin{remark}
\label{rem:higher_dim}
For $n \geq 4$, let $K \subseteq \R^n$ be a symmetric convex body that arises from $\CP^n$ by cutting a single pair of antipodal vertices $\pm v$ off of $\CP^n$ with hyperplanes $\pm H$, where $H$ is sufficiently close to $v$ such that $\B^n \subseteq K$. The resulting convex body $K$ is clearly not affinely equivalent to $\CP^n$ or $\CC^n$ since it has $2n+2$ facets, which is different from $2n$ and $2^n$ for $n \geq 4$. By Theorem~\ref{thm:ader_cond}, there exists an Ader decomposition for $\CP^n$ that is supported on at most $\frac{n (n+1)}{2} + 1$ inner and outer contact points with $\B^n$ and $\sqrt{n} \B^n$ in total. By the proof of Corollary~\ref{cor:estimate} and $d_{BM}(\CP^n,\B^n) = \sqrt{n}$, the matrix underlying this decomposition must be the identity matrix. In other words, there exist two John decompositions supported on the inner and rescaled outer contact points of $\CP^n$ with $\B^n$ and $\sqrt{n} \B^n$, respectively, such that they use at most $\frac{n (n+1)}{2} + 1$ points in total. All inner contact points of $\CP^n$ and $\B^n$ are by $\B^n \subseteq K \subseteq \CP^n$ also inner contact points of $K$ and $\B^n$. Since the John decomposition based on the inner contact points uses at least $n$ points, we know that the decomposition for the rescaled outer contact points uses at most $\frac{n (n+1)}{2} + 1 - n$ points, which is less than $2^{n-1}$ for $n \geq 4$. Therefore, $\CP^n$ has at least one vertex $w$ such that $w$ and $-w$ do not appear in the decomposition for the outer contact points. If we choose $v = w$, then all outer contact points between $\CP^n$ and $\B^n$ that appear in the Ader decomposition for $\CP^n$ are still contact points between $K$ and $\B^n$. Altogether, the Ader decomposition for $\CP^n$ is also an Ader decomposition for $K$, which shows $d_{BM}(K,\B^n) = \sqrt{n}$.

It is clear that there is much more freedom in choosing $K$, as the argument works for every symmetric convex body $K'$ such that $K \subseteq K' \subseteq \CP^n$ for $K$ as constructed above. In particular, we can obtain polytopes with an arbitrarily large number of facets as $K'$, but $K'$ does not even need to be a polytope.
\end{remark}

The last goal of this section is to establish a result attributed to Maurey about the uniqueness of distance ellipsoids. It has been mentioned in multiple different papers (cf.~\cite{tomczakstructure}, \cite[Remark~$1.2$]{arias}, or \cite[Theorem~$3.1$]{praetorius}), but no proof has ever been published.
\newpage

A pair of ellipsoids $E_1, E_2 \subseteq \R^n$ is said to be a \cemph{pair of distance ellipsoids} for a symmetric convex body $K \subseteq \R^n$ if $E_1$ and $E_2$ are homothetic with ratio $d_{BM}(K,\B^n)$ and $E_1 \subseteq K \subseteq E_2$. Note that we do not assume anything about the positions of $E_1$ and $E_2$. In particular, they need not necessarily be concentric. The symmetry of $K$ implies that there always exists at least one pair of distance ellipsoids centered at the symmetry center of $K$.

The strongest version of Maurey's result mentioned in the literature (cf.~\cite[Remark~$1.2$]{arias}) states that if an origin symmetric convex body $K \subseteq \R^n$ has two different pairs of origin concentric distance ellipsoids, then there exists a proper linear subspace $U \subseteq \R^n$ such that $K \cap U$ has a unique pair of origin concentric distance ellipsoids and $d_{BM}(K\cap U, \B^n \cap U) = d_{BM}(K,\B^n)$. Here and in the following, $d_{BM}(K \cap U,\B^n \cap U)$ means the Banach-Mazur distance between the two convex bodies relative to their affine hull $U$. It is not entirely clear whether Maurey's original result concerned only ellipsoids centered at the origin. We shall anyway prove the more general version where the ellipsoids are allowed to be arbitrarily centered.

\begin{twr}
\label{thm:maurey}
Let $K \subseteq \R^n$ be an origin symmetric convex body. Then there exists a linear subspace $U \subseteq \R^n$ (possibly $U = \R^n$) such that $d_{BM}(K \cap U, \B^n \cap U) = d_{BM}(K,\B^n)$ and the pair of distance ellipsoids for $K \cap U$ is unique.
\end{twr}

Before giving the proof, we need some auxiliary results. We start with the following standard fact about polar ellipsoids, which follows immediately from the Cauchy-Schwarz inequality.

\begin{lem}
\label{lem:polar_ellipsoid}
Suppose that vectors $v^1, \ldots, v^n \in \R^n$ form an orthonormal basis, $\alpha_1, \ldots, \alpha_n > 0$ are reals, and $E \subseteq \R^n$ is an ellipsoid defined as
$$E = \left\{ x \in \R^n : \sum_{i=1}^n \frac{\langle x, v^i \rangle^2}{\alpha_i^{2}} \leq 1 \right\}.$$
Then the ellipsoid
$$F = \left\{ y \in \R^n : \sum_{i=1}^n \alpha_i^{2} \langle y, v^i \rangle^2 \leq 1 \right\} $$
is the polar of $E$, and for any $x \in \R^n$ with $\|x\|_E=1$, the unique vector $y \in F$ satisfying $\langle x, y \rangle =1$ (i.e., that a hyperplane perpendicular to $y$ supports $E$ at $x$) is $y = \sum_{i=1}^n \frac{\langle x, v^i \rangle }{\alpha_i^2} v^i$.
\end{lem}
\begin{proof}
For any vectors $x,y \in \R^n$, we have by the Cauchy-Schwarz inequality
$$ \langle x, y \rangle^2 = \left ( \sum_{i=1}^{n} \langle x, v^i \rangle \langle y, v^i \rangle \right )^2 \leq \left ( \sum_{i=1}^n \frac{\langle x, v^i \rangle^2}{\alpha_i^{2}} \right ) \left ( \sum_{i=1}^n \alpha_i^{2} \langle y, v^i \rangle^2 \right ) = \|x\|_E^2 \|y\|_F^2. $$
Hence, the inequality $\langle x, y \rangle \leq \|x\|_E \|y\|_{F}$ is true. Moreover, by the equality condition in the Cauchy-Schwarz inequality, we see for $\| x \|_E = 1$ that equality holds if and only if for some $t \geq 0$, we have $\alpha_i \langle y, v^i \rangle = t \frac{\langle x, v^i \rangle}{\alpha_i}$ for every $i \in \{1, \ldots, n\}$. In this case, $\langle x, y \rangle = 1$ holds precisely when $t=1$, so the conclusion follows.
\end{proof}

The next lemma provides a way of taking means of two different ellipsoids giving the Banach-Mazur distance to the Euclidean ball for a convex body $K \subseteq \R^n$. Clearly, the assumption that one of these ellipsoids is $\B^n$ is not actually restrictive, as this is only a question of applying a suitable affine transformation to $K$. The main benefit provided by the mean ellipsoids is that their contact points with $K$ are reduced to be in a particular subspace (see Figure~\ref{fig:mean_ellipsoid} for an example).

The result is inspired by and can be partly derived from geometric means of ellipsoids (cf.~\cite{brandenberg,milmanrotem}). However, since parts (ii) and (iii) require a closer analysis, we provide a complete proof that does not require familiarity with the notion of geometric means of ellipsoids.

\begin{lem}
\label{lem:mean_ellpsoid}
Let $K \subseteq \R^n$ be a convex body such that $\B^n \subseteq K \subseteq d \B^n$ for some $d \geq 1$. Moreover, suppose that vectors $v^1, \ldots, v^n \in \R^n$ form an orthonormal basis, $\alpha_1, \ldots, \alpha_n>0$ are reals, $\lambda \in [0,1]$ is a real parameter, and the ellipsoid $E_{\lambda} \subseteq \R^n$ is defined as
$$ E_\lambda = \left\{ x \in \R^n : \sum_{i=1}^n \frac{\langle x, v^i \rangle^2}{\alpha_i^{2 \lambda}} \leq 1 \right\}. $$
Let $V = \lin\{ v^i : \alpha_i=1, i=1\ldots,n\}$. If $E_1 \subseteq K \subseteq d E_1$, we have for every $\lambda \in (0,1)$ that
\begin{enumerate}[(i)]
\item $E_\lambda \subseteq K \subseteq d E_\lambda$,
\item $\bd(K) \cap \bd(d E_\lambda) \subseteq V$, and
\item $\bd(K) \cap \bd( E_\lambda) \subseteq V$.
\end{enumerate}
\end{lem}

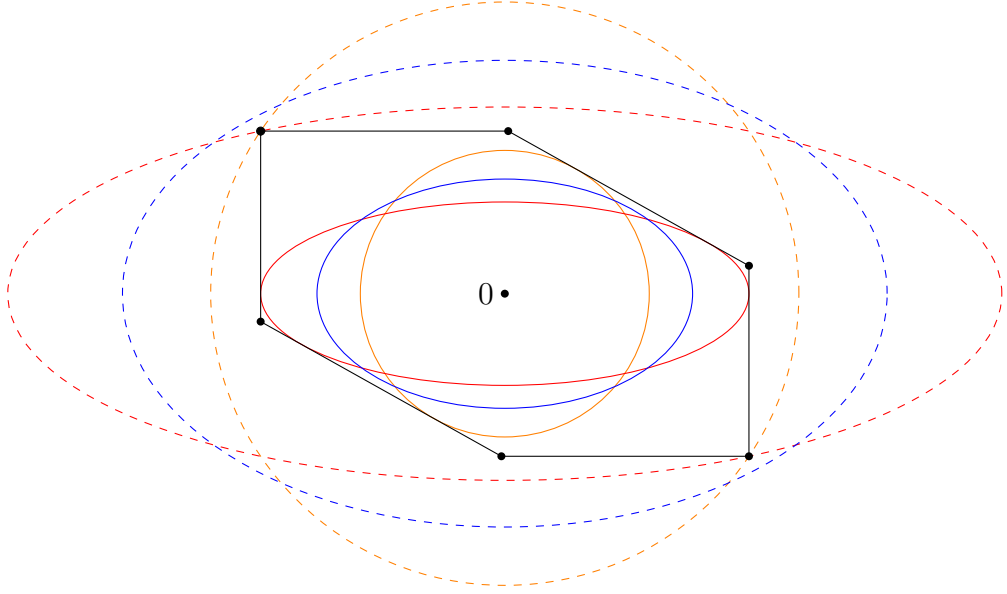
\begin{figure}[ht]
\def\a{1.69}
\def\sa{1.3}
\def\b{0.64}
\def\sb{0.8}
\def\d{sqrt((\a^2-\b^2)/(1-\b^2))}
\def\scale{1.6}
\centering
\begin{tikzpicture}[scale=\scale]
\draw[orange] (0,0) circle(1);
\draw[orange,dashed] (0,0) circle(\d);

\draw[red] (0,0) ellipse (\a cm and \b cm);
\draw[red,dashed] (0,0) ellipse (\d*\a cm and \d*\b cm);

\draw[blue] (0,0) ellipse (\sa cm and \sb cm);
\draw[blue,dashed] (0,0) ellipse (\d*\sa cm and \d*\sb cm);

\draw (-\a,{\b*sqrt((\a^2-1)/(1-\b^2))}) circle(\ds)
        -- (-\a,{-(sqrt(\a^2-\b^2)-\a*sqrt(1-\b^2))/sqrt(\a^2-1)})
        -- ({-(sqrt((\a^2-\b^2)*(1-\b^2))-(\a^2-1)*\b)/(1-\b^2)},
            {-\b*sqrt((\a^2-1)/(1-\b^2))})
        -- (\a,{-\b*sqrt((\a^2-1)/(1-\b^2))})
        -- (\a,{(sqrt(\a^2-\b^2)-\a*sqrt(1-\b^2))/sqrt(\a^2-1)})
        -- ({(sqrt((\a^2-\b^2)*(1-\b^2))-(\a^2-1)*\b)/(1-\b^2)},
            {\b*sqrt((\a^2-1)/(1-\b^2))})
        -- cycle;

\fill ({(sqrt((\a^2-\b^2)*(1-\b^2))-(\a^2-1)*\b)/(1-\b^2)},
            {\b*sqrt((\a^2-1)/(1-\b^2))}) circle(\ds);
\fill ({-(sqrt((\a^2-\b^2)*(1-\b^2))-(\a^2-1)*\b)/(1-\b^2)},
            {-\b*sqrt((\a^2-1)/(1-\b^2))}) circle(\ds);

\fill (\a,{(sqrt(\a^2-\b^2)-\a*sqrt(1-\b^2))/sqrt(\a^2-1)}) circle(\ds);
\fill (-\a,{-(sqrt(\a^2-\b^2)-\a*sqrt(1-\b^2))/sqrt(\a^2-1)}) circle(\ds);

\fill (-\a,{\b*sqrt((\a^2-1)/(1-\b^2))}) circle(\ds);
\fill (\a,{-\b*sqrt((\a^2-1)/(1-\b^2))}) circle(\ds);

\fill (0,0) circle(\ds) node[anchor=east] {$0$};
\end{tikzpicture}
\caption{
An example for Lemma~\ref{lem:mean_ellpsoid}: $K$ (black), $\B^2$ (orange, solid), $E_1$ (red, solid), $E_{1/2}$ (blue, solid). The dashed ellipses are obtained from the solid ellipses by scaling with factor $d \approx 2$. Neither of the principal semi-axes of $E_1$ has length $1$, so $\bd(K)$ is guaranteed to not intersect $\bd(E_\lambda)$ and $\bd(d E_\lambda)$ for any $\lambda \in (0,1)$.
}
\label{fig:mean_ellipsoid}
\end{figure}

\begin{proof} Let us take any vector $x \in K$. By the inequality of weighted arithmetic and geometric means, we have for any $\alpha > 0$ and $\lambda \in (0, 1)$ that
$$ \frac{1}{\alpha^{2 \lambda}} \leq \frac{\lambda}{\alpha^2} + (1-\lambda). $$
Hence,
\begin{align*}
\|x\|_{E_\lambda}^2
& = \sum_{i=1}^n \frac{\langle x, v^i \rangle^2}{\alpha_i^{2 \lambda}}
\leq \sum_{i=1}^n \langle x, v^i \rangle^2 \left( \frac{\lambda}{\alpha_i^2} + (1-\lambda) \right)
= \lambda \sum_{i=1}^n \frac{\langle x, v^i \rangle^2 }{\alpha_i^2} + (1-\lambda) \sum_{i=1}^n \langle x, v^i \rangle^2  \\
& =\lambda \|x\|_{E_1}^2 + (1-\lambda) \|x\|^2
\leq \lambda d^2 + (1-\lambda) d^2
= d^2,
\end{align*}
so $x \in d E_\lambda$. By the equality case in the inequality of arithmetic and geometric means, the equality $\|x\|_{E_\lambda} = d$ holds only if we have $\langle x, v^i \rangle  = 0$ or $\alpha_i = 1$ for every $i \in \{1, \ldots, n\}$. Since the vectors $v^1, \ldots, v^n$ form an orthonormal basis of $\R^n$, we obtain $x = \sum_{i=1}^n \langle x, v^i \rangle v^i \in V$ in this case. As $x \in K$ has been chosen arbitrarily, we have $K \subseteq d E_\lambda$ and $\bd(K) \cap \bd(d E_\lambda) \subseteq V$.

To establish the other inclusion and (iii) from the first part of the proof, it is enough to use a duality argument. Indeed, the inclusions $\B^n \subseteq K$ and $E_1 \subseteq K$ imply that $K^{\circ} \subseteq \B^n$ and $K^{\circ} \subseteq E_1^{\circ}$. Hence, by Lemma~\ref{lem:polar_ellipsoid} and the same reasoning as in the previous part, we get that $K^{\circ} \subseteq E_{\lambda}^{\circ}$, which yields the desired inclusion $E_{\lambda} \subseteq K$. Moreover, if $y \in K^{\circ}$ satisfies $\|y\|_{E_{\lambda}^{\circ}}=1$, then $y \in V$. To prove the third part of the lemma, let us now take $x \in \bd(K) \cap \bd(E_\lambda)$ and a vector $y \in \bd(K^{\circ}) \cap \bd(E^{\circ}_\lambda)$ with $\langle x, y \rangle = 1$. We already know that $y \in V$ and it follows from Lemma~\ref{lem:polar_ellipsoid} that $y = \sum_{i=1}^n \frac{\langle x, v^i \rangle }{\alpha_i^2} v^i$. Thus, if for some $i \in \{1, \ldots, n\}$ we have $\alpha_i \neq 1$, then we must have $\langle x, v^i \rangle = 0$. This shows that $x$ also lies in the subspace $V$, and the conclusion follows.
\end{proof}

The following lemma, which is a straightforward consequence of the Ader decomposition, provides the final, crucial ingredient for our proof of Theorem~\ref{thm:maurey}.

\begin{lem}
\label{lem:contact_space}
Let $K \subseteq \R^n$ be an origin symmetric convex body and let $R \geq r > 0$ be such that $r\B^n \subseteq K \subseteq R\B^n$ and $\frac{R}{r} = d_{BM}(K,\B^n)$. If $U \subseteq \R^n$ is a linear subspace with $\bd(K) \cap \bd(R\B^n) \subseteq U$ or $\bd(K) \cap \bd(r\B^n) \subseteq U$, then $d_{BM}(K \cap U,\B^n \cap U) = d_{BM}(K,\B^n)$.
\end{lem}
\begin{proof}
We may assume that $U \neq \R^n$, as otherwise the result is trivially true. Let an Ader decomposition as in Theorem~\ref{thm:ader_cond}~(iii) be given. Assume that all $y^i$ belong to $U$, but $z^1$ does not. Take any $u \in U^\perp$ such that $\langle u, z^1 \rangle \neq 0$. Then we obtain from the Ader decomposition that
$$ 0 = \sum_{i=1}^N \lambda_i \langle u, y^i \rangle^2 = \sum_{i=1}^M \mu_i \langle u, z^i \rangle^2 \geq \mu_1 \langle u, z^1 \rangle^2 > 0, $$
which is a contradiction. Thus, if all $y^i$ belong to $U$, then so do all $z^i$. Similarly, we conclude that if all $z^i$ belong to $U$, then so do all $y^i$. Altogether, we obtain that there also exists an Ader decomposition for $K \cap U$ relative to $U$. Since the Euclidean in- and circumradius of $K \cap U$ are $r$ and $R$, respectively, we obtain $d_{BM}(K \cap U, \B^n \cap U) = \frac{R}{r} = d_{BM}(K,\B^n)$ as desired.
\end{proof}

As a by-product of the previous lemma, we obtain an optimal estimate on the number of contact points between a symmetric convex body and its distance ellipsoid. We provide the details in the following remark.

\begin{remark}
\label{rem:contact_pairs}
Lewis proved in \cite[Theorem~$2.1$]{lewis} that for any pair of origin symmetric convex bodies in $\R^n$ (where $n \geq 2$), for any linear transformations realizing the Banach-Mazur distance there must exist at least two antipodal pairs of inner and outer contact points. The previous lemma allows us to improve this observation in the case where one of the convex bodies is the Euclidean ball. Indeed, let $K \subseteq \R^n$ be an origin symmetric convex body and let $R \geq r > 0$ be such that $r\B^n \subseteq K \subseteq R\B^n$ and $\frac{R}{r} = d_{BM}(K,\B^n)$. Let $U$ be the linear subspace spanned by the inner (or outer) contact points. The previous lemma and the general upper bound on the distance to the Euclidean ball show that
$$ d_{BM}(K,\B^n) = d_{BM}(K \cap U, \B^n \cap U) \leq \sqrt{\dim(U)}. $$
It follows that for any origin symmetric convex body $K$, the number of linearly independent inner (or outer) contact points is at least $\left\lceil d_{BM}(K, \B^n)^2 \right\rceil$. This estimate is sharp by the example given in \cite[Theorem~$4.3$~(a)]{praetorius}.
\end{remark}

We finally turn to the proof of Theorem~\ref{thm:maurey}.

\begin{proof}[Proof of Theorem~\ref{thm:maurey}]
We proceed by induction on $n$. If $n=1$, the result is trivially true since we can choose the subspace $U = \R^1$. Let us assume that $n \geq 2$. By the origin symmetry of $K$, there exists a pair of origin concentric distance ellipsoids $(E_1,E_2)$ for $K$. If there is no different pair of distance ellipsoids, then there is nothing to prove since we can take $U = \R^n$. Thus, let $(E_1',E_2') \neq (E_1, E_2)$ be a different pair of distance ellipsoids for $K$. By applying an appropriate linear transformation if necessary, we may assume that $E_1'$ and $E_2'$ are Euclidean balls. Let $F_1'$ and $F_2'$ be the origin symmetric translates of $E_1'$ and $E_2'$, respectively. By the origin symmetry of $K$, we have $F_1' \subseteq K \subseteq F_2'$.

First, we show that actually $F_2' = E_2'$. To this end, let $v \in \R^n$ be the center of $E_2'$, i.e., $E_2' = F_2' + v$. If $v\neq 0$, then $K, K + v \subseteq E_2'$ and the strict convexity of $E_2'$ would imply for $x \in K$ that
$$ x + \frac{v}{2} = \frac{1}{2} ( x + (x+v) ) \in \inte(E_2'). $$
However, this would mean that some smaller homothet of $E_2'$ contains $K$, which would contradict the fact that $(E_1',E_2')$ is a pair of distance ellipsoids for $K$. Thus, we have $v=0$ and $F_2' = E_2'$ as claimed.

Next, suppose that $F_1' \neq E_1'$. Let $w \in \R^n \setminus \{0\}$ be the center of $E_1'$, i.e., $E_1' = F_1' + w$. By the origin symmetry of $K$, we clearly have $F_1' - w = -E_1' \subseteq K$. From the inclusions $F_1' - w, F_1' + w \subseteq K$ and the fact that $F_1'$ is a Euclidean ball, it follows that every common boundary point of $K$ and $F_1'$ must lie in the proper linear subspace $U' = \lin\{w\}^\perp$. Indeed, $F_1 \subseteq  \conv( (F_1' - w) \cup (F_1' + w)) \subseteq K$ and it is easy to verify that every common boundary point of $F_1'$ and $\conv( (F_1' - w) \cup (F_1' + w))$ lies in $U'$. Consequently, Lemma~\ref{lem:contact_space} shows that $d_{BM}(K \cap U', \B^n \cap U') = d_{BM}(K,\B^n)$. Applying the induction hypothesis to $K \cap U'$ completes the case $F_1' \neq E_1'$. Thus, we may from now on assume that $F_1' = E_1'$.
\newpage

Finally, since $(F_1',F_2') = (E_1',E_2') \neq (E_1,E_2)$, Lemma~\ref{lem:mean_ellpsoid} shows that there exists another pair of distance ellipsoids for $K$ centered at the origin such that all contact points of $K$ with the new inner and outer ellipsoid lie in some proper subspace $U'' \subsetneq \R^n$. Lemma~\ref{lem:contact_space} thus shows $d_{BM}(K \cap U'',\B^n \cap U'') = d_{BM}(K,\B^n)$. Applying the induction hypothesis to $K \cap U''$ completes the proof.
\end{proof}

We end this section with an immediate consequence of Theorem~\ref{thm:maurey}. A very similar result is also given in \cite[Corollary~$3.2$]{praetorius}, though the statement there is restricted to pairs of origin concentric distance ellipsoids and its proof was based on the at that time still unpublished result by Maurey. See also \cite[Satz~$5$]{behrend} for the planar case, where arbitrary translations of the distance ellipsoids are included.

\begin{cor}
\label{cor:unique}
Let $K \subseteq \R^n$ be a symmetric convex body with $d_{BM}(K,\B^n) > \sqrt{n-1}$. Then, the pair of distance ellipsoids for $K$ is unique. If $d_{BM}(K,\B^n) = \sqrt{n}$, then this unique pair of distance ellipsoids consists of the John and Loewner ellipsoids of $K$. In particular, the pair of distance ellipsoids is unique for any planar symmetric convex body.
\end{cor}

\begin{proof}
Without loss of generality we may assume that $K$ is origin symmetric. Let $U \subseteq \R^n$ be the subspace obtained from Theorem~\ref{thm:maurey}. The general upper bound on the distance to the Euclidean ball shows that
$$ n-1 < d_{BM}(K,\B^n)^2 = d_{BM}(K \cap U, \B^n \cap U)^2 \leq \dim(U), $$
which implies $U = \R^n$. Thus, the pair of distance ellipsoids for $K$ is unique. If the distance is equal to $\sqrt{n}$, then the pair of distance ellipsoids consists of the John and Loewner ellipsoids since they both lead to the upper bound of $\sqrt{n}$ on the Banach-Mazur distance.
\end{proof}

\section{Planar Symmetric Convex Bodies With Almost Maximal Distance to the Euclidean Disc}
\label{sec:2d}

In this section, we prove Theorem~\ref{thm:stability}, that is, a stability of the parallelogram as the unique symmetric convex body with the maximal Banach-Mazur distance to the Euclidean disc. We base our argument on John decompositions instead of Ader decompositions here since knowing the matrix in the decomposition explicitly appears more straightforward to work with. Our main idea is to show for a symmetric convex body $K \subseteq \R^2$ with $d_{BM}(K,\B^2)$ close to $\sqrt{2}$ and John ellipse $\B^2$ that any John decomposition contains a pair of almost orthogonal vectors. From this, we derive that $K$ must be close to a certain square. In other words, we provide a stability version of Lemma~\ref{lem:n_John_decomp} in the planar case. The following lemma is the key to executing this idea and obtaining the linear upper bound in the final stability estimate.

\begin{lem}
\label{lem:orth}
Let $K \subseteq \R^2$ be an origin symmetric convex body with John ellipse $\B^2$. For $v \in K \setminus \{ 0 \}$, let $e^1, e^2$ be the two orthogonal Euclidean unit vectors satisfying $\langle v, e^1 \rangle = \langle v, e^2 \rangle = \frac{\|v\|}{\sqrt{2}}$. Then for any $x \in K$ and $i = 1,2$, we have
$$|\langle x, e^i \rangle | \leq \frac{\sqrt{2}}{\|v\|}.$$
\end{lem}

For obtaining the stability result, we are mostly interested in the case when $\|v\|$ is close to $\sqrt{2}$. However, let us point out that the above inequality is tight for any value of $\| v \| \geq 1$, which can be seen by choosing $K = \CP^2$ and $v \in \bd(\CP^2)$.

\begin{proof}
If $d := \|v\| \leq 1$, then the assertion follows from the Cauchy-Schwarz inequality together with the inclusion $K \subseteq \sqrt{2} \B^2$ following from the John Ellipsoid Theorem. We may thus assume that $d > 1$.

By applying an appropriate rotation if necessary, we may assume $e^1 = (1,0)$, $e^2 = (0,1)$, and $v = \frac{d}{\sqrt{2}} (1,1)$. In this case, the two vectors
$$ p^\pm := \frac{1}{\sqrt{2} d} \left( 1 \pm \sqrt{d^2 - 1}, 1 \mp \sqrt{d^2 - 1} \right) $$
are the only vectors $p \in \bd(\B^2)$ satisfying $\langle p, v \rangle = 1$, i.e., $p^{\pm}$ are tangency points of tangents from $v$ to $\B^2$. Defining the two lines
$$ L^\pm := \left\{ x \in \R^2 : \langle p^\pm, x \rangle = 1 \right\}, $$
one can easily verify that
$$ a^\pm := \frac{\sqrt{2}}{d} \left( 1, \mp \sqrt{d^2 - 1} \right)
\quad \text{and} \quad
b^\pm := \frac{\sqrt{2}}{d} \left( \pm \sqrt{d^2 - 1}, 1 \right) $$
satisfy
$$ L^\pm \cap \bd \left( \sqrt{2} \B^2 \right) = \left\{ a^\pm, b^\pm \right\} $$
(cf.~Figure~\ref{fig:stab_lem_points}).

\begin{figure}[ht]
\def\d{1.15}
\def\scale{2}
\centering
\begin{tikzpicture}[scale=\scale]
\draw circle(1);
\draw circle({sqrt(2)});

\draw ({sqrt(2)/\d},{-sqrt(2)/\d*sqrt(\d^2-1)}) -- ({sqrt(2)/\d*sqrt(\d^2-1)},{sqrt(2)/\d});
\draw ({sqrt(2)/\d},{sqrt(2)/\d*sqrt(\d^2-1)}) -- ({-sqrt(2)/\d*sqrt(\d^2-1)},{sqrt(2)/\d});

\fill ({\d/sqrt(2)},{\d/sqrt(2)}) circle(\ds) node[anchor=south west] {$v$};

\fill (1,0) circle(\ds) node[anchor=east] {$e^1$};
\fill (0,1) circle(\ds) node[anchor=north] {$e^2$};

\fill ({sqrt(2)/\d},{-sqrt(2)/\d*sqrt(\d^2-1)}) circle(\ds) node[anchor=north west] {$a^+$};
\fill ({sqrt(2)/\d},{sqrt(2)/\d*sqrt(\d^2-1)}) circle(\ds) node[anchor=south west] {$a^-$};

\fill ({sqrt(2)/\d*sqrt(\d^2-1)},{sqrt(2)/\d}) circle(\ds) node[anchor=south west] {$b^+$};
\fill ({-sqrt(2)/\d*sqrt(\d^2-1)},{sqrt(2)/\d}) circle(\ds) node[anchor=south east] {$b^-$};

\fill ({(1+sqrt(\d^2-1))/(sqrt(2)*\d)},{(1-sqrt(\d^2-1))/(sqrt(2)*\d)}) circle(\ds) node[anchor=east] {$p^+$};
\fill ({(1-sqrt(\d^2-1))/(sqrt(2)*\d)},{(1+sqrt(\d^2-1))/(sqrt(2)*\d)}) circle(\ds) node[anchor=north] {$p^-$};

\fill (0,0) circle(\ds) node[anchor=south west] {$0$};
\end{tikzpicture}
\caption{
An example of the situation in the proof of Lemma~\ref{lem:orth} for $d = \d$.
}
\label{fig:stab_lem_points}
\end{figure}

Since $\B^2$ is the John ellipsoid of $K$, there exist common boundary points $u^1, \ldots, u^N$ of $K$ and $\B^2$ and some weights $\lambda_1, \ldots, \lambda_N > 0$ that form a John decomposition. In this case, we have $\sum_{i=1}^N \lambda_i = 2$.

We observe for any $w \in \bd(\sqrt{2} \B^2)$ that there exists at least one index $i \in \{1, \ldots, N\}$ such that $|\langle w, u^i \rangle| \geq 1$. Indeed, we would otherwise have
$$ 2 = \langle w, w \rangle = \sum_{i=1}^m \lambda_i \langle w, u^i \rangle^2 < \sum_{i=1}^m \lambda_i = 2. $$
Furthermore, if $\theta: \R^2 \to \R^2$ is the counterclockwise rotation by $90^\circ$, then $p^+$ and $\theta(p^+)$ are the only vectors $q \in \bd(\B^2)$ with $\langle q, b^+ \rangle = 1$.

Now, for integer $k \geq 1$ let $\psi_k: \R^2 \to \R^2$ be the counterclockwise rotation by $\frac{1}{k}$ degrees. For any $w^k := \psi_k(b^+)$, the set of vectors $q \in \bd(\B^2)$ with $\langle q, w^k \rangle \geq 1$ is precisely the minor circular arc connecting $\psi_k(p^+)$ and $\psi_k(\theta(p^+))$. It is clear that any such vector $q \neq p^-$ that also lies on the minor circular arc connecting $p^+$ and $p^-$ is in the interior of $\conv(\{v\} \cup \B^2)$ and thus also in the interior of $K$ (cf.~Figure~\ref{fig:stab_lem_rot}). Hence, for any $k \in \N$, there must be, by the above observation applied for $w^k$, at least one contact point $\pm u^i$ that lies on the minor circular arc connecting $p^-$ and $\psi_k(\theta(p^+))$. Since there are only finitely many contact points $\pm u^i$, by letting $k$ tend to $\infty$ we see that at least one of the $\pm u^i$ must lie on the minor circular arc connecting $\theta(p^+)$ and $p^-$. Let us suppose that this is the case for $u^1$.

\begin{figure}[ht]
\def\d{1.15}
\def\deg{20}
\def\scale{3}
\centering
\begin{tikzpicture}[scale=\scale]

\draw (1,0) arc (0:180:1);
\draw ({sqrt(2)},0) arc (0:180:{sqrt(2)});

\draw ({(sqrt(2)*\d))/(1+sqrt(\d^2-1))},0) -- ({sqrt(2)/\d*sqrt(\d^2-1)},{sqrt(2)/\d});
\draw ({sqrt(2)/\d},{sqrt(2)/\d*sqrt(\d^2-1)}) -- ({-sqrt(2)/\d*sqrt(\d^2-1)},{sqrt(2)/\d});
\draw[dashed] ({sqrt(2)/\d*sqrt(\d^2-1)},{sqrt(2)/\d}) -- ({-sqrt(2)/\d},{sqrt(2)/\d*sqrt(\d^2-1)});

\draw[very thick,dotted,red]
    ({cos(\deg)*(1+sqrt(\d^2-1))/(sqrt(2)*\d)-sin(\deg)*(1-sqrt(\d^2-1))/(sqrt(2)*\d)},
     {sin(\deg)*(1+sqrt(\d^2-1))/(sqrt(2)*\d)+cos(\deg)*(1-sqrt(\d^2-1))/(sqrt(2)*\d)})
    arc ({atan((1-sqrt(\d^2-1))/(1+sqrt(\d^2-1)))+\deg}:{atan((1+sqrt(\d^2-1))/(1-sqrt(\d^2-1)))}:1);

\draw[very thick,red]
    ({(1-sqrt(\d^2-1))/(sqrt(2)*\d)},{(1+sqrt(\d^2-1))/(sqrt(2)*\d)})
    arc ({atan((1+sqrt(\d^2-1))/(1-sqrt(\d^2-1)))}:{atan((1-sqrt(\d^2-1))/(1+sqrt(\d^2-1)))+\deg+90}:1);

\draw[dashed,red] ({cos(\deg)*sqrt(2)/\d*sqrt(\d^2-1)-sin(\deg)*sqrt(2)/\d},{sin(\deg)*sqrt(2)/\d*sqrt(\d^2-1)+cos(\deg)*sqrt(2)/\d})
            -- ({cos(\deg)*(1+sqrt(\d^2-1))/(sqrt(2)*\d)-sin(\deg)*(1-sqrt(\d^2-1))/(sqrt(2)*\d)},
                {sin(\deg)*(1+sqrt(\d^2-1))/(sqrt(2)*\d)+cos(\deg)*(1-sqrt(\d^2-1))/(sqrt(2)*\d)});

\draw[dashed,red] ({cos(\deg)*sqrt(2)/\d*sqrt(\d^2-1)-sin(\deg)*sqrt(2)/\d},{sin(\deg)*sqrt(2)/\d*sqrt(\d^2-1)+cos(\deg)*sqrt(2)/\d})
            -- ({-cos(\deg)*(1-sqrt(\d^2-1))/(sqrt(2)*\d)-sin(\deg)*(1+sqrt(\d^2-1))/(sqrt(2)*\d)},
                {-sin(\deg)*(1-sqrt(\d^2-1))/(sqrt(2)*\d)+cos(\deg)*(1+sqrt(\d^2-1))/(sqrt(2)*\d)});

\fill ({\d/sqrt(2)},{\d/sqrt(2)}) circle(\ds) node[anchor=south west] {$v$};

\fill (0,1) circle(\ds) node[anchor=north] {$e^2$};

\fill ({sqrt(2)/\d},{sqrt(2)/\d*sqrt(\d^2-1)}) circle(\ds) node[anchor=south west] {$a^-$};

\fill ({sqrt(2)/\d*sqrt(\d^2-1)},{sqrt(2)/\d}) circle(\ds) node[anchor=south west] {$b^+$};
\fill ({-sqrt(2)/\d*sqrt(\d^2-1)},{sqrt(2)/\d}) circle(\ds) node[anchor=south east] {$b^-$};

\fill ({(1+sqrt(\d^2-1))/(sqrt(2)*\d)},{(1-sqrt(\d^2-1))/(sqrt(2)*\d)}) circle(\ds) node[anchor=east] {$p^+$};
\fill ({(1-sqrt(\d^2-1))/(sqrt(2)*\d)},{(1+sqrt(\d^2-1))/(sqrt(2)*\d)}) circle(\ds) node[anchor=north] {$p^-$};
\fill ({-(1-sqrt(\d^2-1))/(sqrt(2)*\d)},{(1+sqrt(\d^2-1))/(sqrt(2)*\d)}) circle(\ds) node[anchor=north] {$\theta(p^+)$};

\fill ({cos(\deg)*sqrt(2)/\d*sqrt(\d^2-1)-sin(\deg)*sqrt(2)/\d},{sin(\deg)*sqrt(2)/\d*sqrt(\d^2-1)+cos(\deg)*sqrt(2)/\d})
		circle(\ds) node[anchor=south west] {$w$};

\fill (0,0) circle(\ds) node[anchor=south west] {$0$};
\end{tikzpicture}
\caption{
An example of the situation in the proof of Lemma~\ref{lem:orth} for $d = \d$:
The vector $w$ is constructed like $w^k$ but using the angle $\deg^\circ$ instead of $\frac{1}{k}^\circ$. One of the $\pm u^i$ must lie on the solid red circular arc.
}
\label{fig:stab_lem_rot}
\end{figure}

Since $u^1$ lies on the minor circular arc connecting
$$ \theta(p^+) = \left( - \frac{1 - \sqrt{d^2 - 1}}{\sqrt{2} d}, \frac{1 + \sqrt{d^2 - 1}}{\sqrt{2} d} \right)
\quad \text{and} \quad
p^- = \left( \frac{1 - \sqrt{d^2 - 1}}{\sqrt{2} d}, \frac{1 + \sqrt{d^2 - 1}}{\sqrt{2} d} \right), $$
we immediately see
$$ | u^1_1 | \leq \frac{1 - \sqrt{d^2 - 1}}{\sqrt{2} d}
\quad \text{and} \quad
u^1_2 \geq \frac{1 + \sqrt{d^2 - 1}}{\sqrt{2} d}. $$
Now, assume for a contradiction that there exists some $x \in K$ with $|x_2| = |\langle x, e^2 \rangle| > \frac{\sqrt{2}}{d}$. By symmetry of $K$, we may assume
$$ x_2 > \frac{\sqrt{2}}{d}. $$
Since $K \subseteq \sqrt{2} \B^2$, it follows that
$$ |x_1| < \frac{\sqrt{2}}{d} \sqrt{d^2 - 1}. $$
However, we now obtain
$$ \langle x, u^1 \rangle \geq x_2 u^1_2 - |x_1 u^1_1| > \frac{1 + \sqrt{d^2 - 1}}{d^2} - \frac{1 - \sqrt{d^2 - 1}}{d^2} \sqrt{d^2 - 1} = 1, $$
contradicting $u^1 \in K^\circ$. Hence, every $x \in K$ satisfies $|\langle x, e^2 \rangle| \leq \frac{\sqrt{2}}{d}$ as claimed. The second inequality $|\langle x, e^1 \rangle| \leq \frac{\sqrt{2}}{d}$ follows from symmetry of the situation with respect to reflection at the line spanned by the vector $v$.
\end{proof}

To obtain the linear bound in the stability estimate, we require the following technical lemma.

\begin{lem}
\label{lem:stab_poly_ineq}
For all $r \in [0.95,1)$, the inequality
$$ \frac{\sqrt{2}}{r} \cdot \frac{x(r) + y(r) \sqrt{x(r)^2 + y(r)^2 - 1}}{x(r)^2 + y(r)^2} < 1 + 10 (1 - r) $$
is true, where the functions $x, y \colon [0.95,1) \to \R$ are defined as
$$ x(r) = \frac{\sqrt{2} (2 r^6 - 1)}{r (r^6 + r^2 - 1)}
\quad \text{and} \quad
y(r) = \frac{\sqrt{2} r (1 - r^4)}{r^6 + r^2 - 1}. $$
\end{lem}

Let us briefly observe that the left-hand side in the above inequality is indeed real valued since $x(r) > 1$ for the given range of $r$.

\begin{proof}
We can rewrite the left-hand side as
$$\frac{(2 r^6 - 1) (r^6 + r^2 - 1) + r (1 - r^4) \sqrt{2 \left( 2 r^6 - 1 \right)^2 + 2 r^4 (1 - r^4)^2 - r^2 (r^6 + r^2 - 1)}}{(2 r^6 - 1)^2 + r^4 (1 - r^4)^2}.$$
The claimed inequality is therefore equivalent to
\begin{align}
\begin{split}
\label{eq:poly_ineq_pre_square}
& r (1 - r^4) \sqrt{2 \left( 2 r^6 - 1 \right)^2 + 2 r^4 (1 - r^4)^2 - r^2 (r^6 + r^2 - 1)^2} \\
& < (1 + 10 (1 - r)) ((2 r^6 - 1)^2 + r^4 (1 - r^4)^2) - (2 r^6 - 1) (r^6 + r^2 - 1).
\end{split}
\end{align}
We want to square both sides to eliminate the square root, for which we first need to show that the above right-hand side is positive. To this end, it suffices to prove
$$ (1 + 10 (1 - r)) (2 r^6 - 1)^2 > (2 r^6 - 1) (r^6 + r^2 - 1), $$
which can by $2 r^6 - 1 > 0.4$ for $r \geq 0.95$ be rearranged to
$$ 10 (1 - r) (2 r^6 - 1) > r^2 (1 - r^4) = (1 - r) (r^5 + r^4 + r^3 + r^2). $$
Now, the left-hand side is larger than $4 (1-r)$, whereas the right-hand side is at most $4 (1-r)$. Altogether, \eqref{eq:poly_ineq_pre_square} is equivalent to the inequality with both sides squared, i.e., to the term
\begin{align*}
&\left( (1 + 10 (1 - r)) ((2 r^6 - 1)^2 + r^4 (1 - r^4)^2) - (2 r^6 - 1) (r^6 + r^2 - 1) \right)^2 \\
&- r^2 (1 - r^4)^2 \left( 2 (2 r^6 - 1)^2 + 2 r^4 (1 - r^4)^2 - r^2 (r^6 + r^2 - 1)^2 \right)
\end{align*}
being positive for all $r \in [0.95,1)$. It can be verified by a direct computation that the above expression can be factorized as
$$ 2 (r-1)^2 f(r) g(r),$$
where
$$ f(r) = 250 r^{12} - 30 r^{11} - 29 r^{10} - 28 r^9 - 128 r^8 + 12 r^7 - 190 r^6 + 18 r^5 + 68 r^4 + 8 r^3 + 9 r^2 + 50 $$
and
$$ g(r) = 5 r^{12} - 2 r^8 - 4 r^6 + r^4 + 1$$
It is thus enough to prove that both polynomials $f$ and $g$ are positive on the interval $[0.95, 1]$. We shall rely on the following general observation: if $u \colon \R \to \R$ is a differentiable function such that for some $a \in \R$ we have $u'(r) \geq 0$ for every $r \geq a$ and also $u(a) > 0$, then $u(r) > 0$ for every $r \geq a$.

Let us start with the polynomial $g$. We calculate that
$$g'(r)= 4 r^3 (15 r^8 - 4 r^4 - 6 r^2 + 1), $$
which is positive for $r \in [0.95,1)$ by $15 \cdot 0.95^8 - 4 - 6 + 1 > 0.9$. Therefore, $g$ is increasing on this interval with $g(0.95) > 0.2$, which shows that it is positive for all $r \in [0.95,1)$, as claimed.

Now, we shall prove that $f$ is positive for $r \in [0.95, 1)$. This case requires a more delicate analysis, as $f$ is not monotonic on this interval. The first derivative of $f$ is
$$ 2 r (1500 r^{10} - 165 r^9 - 145 r^8 - 126 r^7 - 512 r^6 + 42 r^5 - 570 r^4 + 45 r^3 + 136 r^2 + 12 r + 9) $$
and the second derivative can be computed to be
$$ 2 (16500 r^{10} - 1650 r^9 - 1305 r^8 - 1008 r^7 - 3584 r^6 + 252 r^5 - 2850 r^4 + 180 r^3 + 408 r^2 + 24 r + 9). $$
Lower estimating $r$ by $0.95$ and upper estimating $r$ by $1$ in all terms with positive resp.~negative coefficient shows that the second derivative is for $r \in [0.95,1)$ always larger than $400$, i.e., positive. Therefore, the first derivative is increasing on this interval. Since $f(0.96)$ and $f'(0.96)$ are both positive, we obtain that $f$ is positive on $[0.96,1)$.

We are left with with the case of $r \in [0.95, 0.96]$. We calculate that $f(0.96)>1$ and, because $f'$ is increasing on this interval, we have
$$ M := \max\{|f'(r)|: \ r \in [0.95, 0.96]\} = \max\{|f'(0.95)|, |f'(0.96)|\} < 36 $$
again by a direct evaluation. Therefore, the mean value theorem shows
$$ |f(0.96) - f(r)| \leq |r-0.96| \cdot M \leq \frac{M}{100} < 1. $$
As $f(0.96)>1$, this yields $f(r) > 0$ and the proof is finished.
\end{proof}

We are now ready to prove the stability estimate.

\begin{proof}[Proof of Theorem~\ref{thm:stability}]
We may assume
$$ r := \frac{d_{BM}(K,\B^2)}{\sqrt{2}} < 1 $$
by Theorem~\ref{thm:eqcase}, and additionally $\varepsilon \leq \frac{\sqrt{2}}{20}$, i.e., $r \geq 0.95$, as otherwise a result of Stromquist \cite{stromquist} shows that
$$ d_{BM}(K,\CP^2) \leq \frac{3}{2} < 1 + \frac{10}{\sqrt{2}} \varepsilon. $$

By applying a suitable affine transformation, we can further assume that $K$ is origin symmetric with John ellipse $\B^2$. In this case, $K$ cannot be contained in the interior of $r \sqrt{2} \B^2$, so by convexity there exists a vector $v \in K$  with $\|v\| = r \sqrt{2}$. For this vector $v$, we take an orthonormal basis $e^1, e^2 \in \B^2$ as in Lemma~\ref{lem:orth}. By applying an appropriate rotation, we may assume $e^1 = \frac{1}{\sqrt{2}} (1, 1)$ and $e^2 = \frac{1}{\sqrt{2}} (-1, 1)$. Then the point $v = (0, r \sqrt{2})$ lies in $K$. From Lemma~\ref{lem:orth} it follows that 
\begin{equation}
\label{eq:stab_inclusion1}
K \subseteq \frac{\sqrt{2}}{r} \CC^2.
\end{equation}

Next, we claim that there exists some $w \in K$ with
\begin{equation}
\label{eq:x(r)}
w_1 \geq \frac{\sqrt{2} (2 r^6 - 1)}{r (r^6 + r^2 - 1)} =: x(r).
\end{equation}
We remark that it is easy to verify that $x(r) \in \left( 1,r \sqrt{2} \right)$ for $r \in [0.95,1)$. Now, let us assume that the above condition does not hold for any $w \in K$. In this case, the inclusion \eqref{eq:stab_inclusion1} and the origin symmetry of $K$ show that $K$ is a subset of the hexagon $H$ with vertex set
$$ \frac{\sqrt{2}}{r} \left\{ ( 0, \pm 1 ), \left( \pm \frac{2 r^6 - 1}{r^6 + r^2 - 1}, \pm \frac{r^2 - r^6}{r^6 + r^2 - 1} \right) \right\} $$
(cf.~Figure~\ref{fig:stab_T(H)}). By construction, all vertices of $H$ lie on the sides of $\frac{\sqrt{2}}{r} \CC^2$, but none of them are contained in $K$ by $K \subseteq \sqrt{2} \B^2 \subseteq \inte( \frac{\sqrt{2}}{r} \B^2)$ and the assumption that \eqref{eq:x(r)} does not hold for any point in $K$.

\begin{figure}[ht]
\def\d{1.35}
\def\r{(\d/sqrt(2))}
\def\xr{(sqrt(2)*(2*\r^6-1)/(\r*(\r^6+\r^2-1)))}
\def\a{(\r^2*sqrt((2*\r^2-1)/(2*\r^6-1)))}
\def\b{(\d^2/2)}
\def\sq{(sqrt(2))}
\def\scale{2.0}
\centering
\begin{tikzpicture}[scale=\scale]
\filldraw[draw=yellow,fill=yellow!50] (0,{2/\d}) -- ({-\xr},{2/\d-\xr}) -- ({-\xr},{\xr-2/\d}) -- (0,{-2/\d}) -- ({\xr},{\xr-2/\d}) -- ({\xr},{2/\d-\xr}) -- cycle;

\draw[dashed] (0,{\d}) -- ({-\a*\xr},{\b*(2/\d-\xr)}) -- ({-\a*\xr},{\b*(\xr-2/\d)}) -- (0,{-\d}) -- ({\a*\xr},{\b*(\xr-2/\d)}) -- ({\a*\xr},{\b*(2/\d-\xr)}) -- cycle;

\draw (0,0) circle(1);
\draw (0,0) circle(\sq);
\draw[dotted] (0,0) circle(\d);

\draw[red] ({2/\d},0) -- (0,{2/\d}) -- ({-2/\d},0) -- (0,{-2/\d}) -- cycle;

\fill (0,{2/\d}) circle(\ds) node[anchor=south] {$\left( 0,\frac{\sqrt{2}}{r} \right)$};
\fill ({1/\sq},{1/\sq}) circle(\ds) node[anchor=north east] {$e^1$};
\fill ({-1/\sq},{1/\sq}) circle(\ds) node[anchor=north west] {$e^2$};
\fill (0,{\d}) circle(\ds) node[anchor=north] {$v$};
\fill (0,0) circle(\ds) node[anchor=south] {$0$};
\end{tikzpicture}
\caption{
An example of the situation in the proof of Theorem~\ref{thm:stability} for $d_{BM}(K,\B^2) \geq \d$: $\frac{\sqrt{2}}{r} \CC^2$ (red), $H$ (yellow), $T(H)$ (dashed), $\B^2$ and $\sqrt{2} \B^2$ (black), $\d \B^2$ (dotted). Since $K \subseteq H$ and none of the vertices of $H$ belong to $K$, we have $T(K) \subseteq \inte(\d \B^2)$.
}
\label{fig:stab_T(H)}
\end{figure}

Our goal is to obtain a contradiction to the fact that no ellipse yields a better bound on the Banach-Mazur distance to $K$ than $r \sqrt{2}$. To this end, we define parameters
$$ \alpha := r^2 \sqrt{\frac{2 r^2 - 1}{2 r^6 - 1}}
\quad \text{and} \quad
\beta := r^2 $$
and a linear transformation $T \colon \R^2 \to \R^2$ by
$$ T(x,y) = (\alpha x, \beta y). $$
Then clearly $T \left( 0,\frac{\sqrt{2}}{r} \right) = \left( 0,r \sqrt{2} \right) \in \bd(r \sqrt{2} \B^2)$. Moreover,
\begin{align*}
\left\| \frac{\sqrt{2}}{r} T \left( \pm \frac{2 r^6 - 1}{r^6 + r^2 - 1}, \pm \frac{r^2 - r^6}{r^6 + r^2 - 1} \right) \right\|^2
& = \frac{2}{r^2} \left( \alpha^2 \frac{(2 r^6 - 1)^2}{(r^6 + r^2 - 1)^2} + \beta^2 \frac{(r^2 - r^6)^2}{(r^6 + r^2 - 1)^2} \right) \\
& = 2 r^2 \frac{(2 r^2 - 1) (2 r^6 - 1) + (r^2 - r^6)^2}{(r^6 + r^2 - 1)^2}
= 2 r^2,
\end{align*}
so $T(H) \subseteq r \sqrt{2} \B^2$. Since none of the vertices of $T(H)$ lie in $T(K)$ and $r \sqrt{2} \B^2$ is strictly convex, it follows that $T(K) \subseteq \inte(r \sqrt{2} \B^2)$.

Moreover, we claim that $\B^2 \subseteq T(K)$ (cf.~Figure~\ref{fig:stab_inner_circle}). To this end, we first prove for
$$ u := \left( \sqrt{1 - \frac{1}{2 r^2}}, \frac{1}{r \sqrt{2}} \right) \in \bd(\B^2) $$
and $S := [v,u] \subseteq K$ that $T(S)$ does not meet $\inte(\B^2)$. The vector
$$ a := \left( \sqrt{1 - \frac{1}{2 \beta^2 r^2}}, \frac{1}{\beta r \sqrt{2}} \right) \in \bd(\B^2) $$
is well-defined by $\beta r \sqrt{2} \geq 0.95^3 \sqrt{2} > 1$, while satisfying $\langle a, T(v) \rangle = \langle a, (0,\beta r \sqrt{2}) \rangle = 1$ and
$$ \left\langle a, T(u) \right\rangle = \alpha \sqrt{\frac{(2 \beta^2 r^2 - 1) (2 r^2 - 1)}{4 \beta^2 r^4}} + \frac{1}{2 r^2} = \frac{2 r^2 - 1}{2 r^2} + \frac{1}{2 r^2} = 1. $$
Hence, the segment $T(S)$ belongs to the line tangential to $\B^2$ at $a$, so $T(S)$ indeed does not meet $\inte(\B^2)$. The part of $\B^2$ above $x_2 = \frac{r}{\sqrt{2}}$ is thus contained in $T(\conv(S \cup S')) \subseteq T(K)$, where $S'$ is the reflection of $S$ at the $y$-axis. Next, we observe that $\alpha > \beta$, so the function $[0,\frac{1}{r \sqrt{2}}] \ni y \mapsto \| T(\sqrt{1-y^2},y) \|$ is decreasing with minimum $\| T(u) \|$ attained for $y = \frac{1}{r \sqrt{2}}$. By the above, we have $\| T(u) \| \geq 1$. Therefore, a slice of $\B^2$ at $0 \leq x_2 \leq \frac{r}{\sqrt{2}}$ is contained in the image of the slice at $\frac{x_2}{r^2}$ under $T$. Altogether, we conclude from the origin symmetry of $T(K)$ that $\B^2 \subseteq T(K)$, contradicting the assumption $r \sqrt{2} = d_{BM}(K,\B^2)$ as desired.

\begin{figure}[ht]
\def\d{1.35}
\def\r{(\d/sqrt(2))}
\def\xr{(sqrt(2)*(2*\r^6-1)/(\r*(\r^6+\r^2-1)))}
\def\a{(\r^2*sqrt((2*\r^2-1)/(2*\r^6-1)))}
\def\b{(\d^2/2)}
\def\sq{(sqrt(2))}
\def\scale{3}
\centering
\begin{tikzpicture}[scale=\scale]
\draw[dashdotted] ({-\sq},{1/\d}) -- ({\sq},{1/\d});
\draw[dotted] ({-\sq},{\d/2}) -- ({\sq},{\d/2});

\draw (1,0) arc (0:180:1);
\draw ({\sq},0) arc (0:180:{\sq});

\draw[blue] (0,{\d}) -- ({sqrt(1-1/\d^2)},{1/\d});
\draw[blue] (0,{\d}) -- ({-sqrt(1-1/\d^2)},{1/\d});

\draw[dashed] (0,{\d^3/2}) -- ({\a*sqrt(1-1/\d^2)},{\d/2});
\draw[dashed] (0,{\d^3/2}) -- ({-\a*sqrt(1-1/\d^2)},{\d/2});

\draw[dashed] ({\a},0) arc [start angle=0, end angle={atan(1/sqrt(\d^2-1))}, x radius={\a}, y radius={\b}];
\draw[dashed] ({-\a*sqrt(1-1/\d^2)},{\d/2}) arc [start angle={180-atan(1/sqrt(\d^2-1))}, end angle=180, x radius={\a}, y radius={\b}];

\fill ({sqrt(1-1/\d^2)},{1/\d}) circle(\ds);
\node[anchor=west] at ({sqrt(1-1/\d^2)-0.02},{1/\d+0.06}) {$u$};
\fill ({-sqrt(1-1/\d^2)},{1/\d}) circle(\ds);

\fill ({\a*sqrt(1-1/\d^2)},{\d/2}) circle(\ds);
\node[anchor=east] at ({\a*sqrt(1-1/\d^2)+0.05},{\d/2-0.1}) {$T(u)$};
\fill ({-\a*sqrt(1-1/\d^2)},{\d/2}) circle(\ds);

\fill ({sqrt(1-1/(2*\r^6))},{1/(\r^3*sqrt(2))}) circle(\ds);
\node[anchor=east] at ({sqrt(1-1/(2*\r^6))-0.025},{1/(\r^3*sqrt(2))-0.01}) {$a$};
\fill ({-sqrt(1-1/(2*\r^6))},{1/(\r^3*sqrt(2))}) circle(\ds);

\fill (0,{\d}) circle(\ds);
\node[anchor=south] at (0,{\d+0.05}) {$v$};
\fill (0,{\d^3/2}) circle(\ds);
\node[anchor=north] at (0,{\d^3/2-0.04}) {$T(v)$};
\fill (0,0) circle(\ds) node[anchor=south] {$0$};
\end{tikzpicture}
\caption{
An example of the situation in the proof of Theorem~\ref{thm:stability} for $d_{BM}(K,\B^2) \geq \d$: $\B^2$ and $\sqrt{2} \B^2$ (black), lines at $x_2 = \frac{1}{r \sqrt{2}}$ (dash-dotted) and $x_2 = \frac{r}{\sqrt{2}}$ (dotted), $S$ and its reflection $S'$ at the $y$-axis (blue), $T(\conv(\{v\} \cup \B^2))$ (dashed). The parameters $\alpha$ and $\beta$ are chosen so that $\frac{\sqrt{2}}{r} \beta = r \sqrt{2}$ and $\alpha > 0$ is smallest possible with $\B^2 \subseteq T(\conv(\{v,-v\} \cup \B^2))$.
}
\label{fig:stab_inner_circle}
\end{figure}

Therefore, there must exist a point $w \in K$ satisfying \eqref{eq:x(r)}. By reflecting everything at the $x$-axis if necessary, we may assume $w_2 \geq 0$. The inclusion \eqref{eq:stab_inclusion1} yields now
\begin{equation}
\label{eq:y(r)}
w_2 \leq \frac{\sqrt{2}}{r} - w_1 \leq \frac{\sqrt{2}}{r} - x(r) = \frac{\sqrt{2} r (1 - r^4)}{r^6 + r^2 - 1} =: y(r).
\end{equation}
For $r \in [0.95, 1)$ we clearly have $y(r)<1$. Let us now define a point
$$ p := \frac{1}{w_1^2 + w_2^2} \left( w_1 + w_2 \sqrt{w_1^2+w_2^2-1}, w_2 - w_1 \sqrt{w_1^2+w_2^2-1} \right), $$
which satisfies $\langle p, p \rangle = \langle p, w \rangle = 1$. Thus, $p \in \B^2 \subseteq K$ is a tangency point of a tangent from $w$ to $\B^2$. Since $w_1 \geq x(r) > 1$ and thus $w_2 = \sqrt{w_2^2} < w_1 \sqrt{w_1^2 + w_2^2 - 1}$, we have
$$ \lambda := \frac{w_2 + w_2 (w_1^2 + w_2^2 - 1)}{w_1 \sqrt{w_1^2 + w_2^2 - 1} + w_2 (w_1^2 + w_2^2 - 1)} \in [0,1] $$
and therefore
$$ q := \left( \frac{w_1^2 + w_2^2}{w_1 + w_2 \sqrt{w_1^2 + w_2^2 - 1}}, 0 \right) = \lambda p + (1-\lambda) w \in \conv(\B^2 \cup \{w\}) \subseteq K.$$
As the point $q$ is the intersection point of the $x$-axis and the tangent of $w$ to $\B^2$ that meets the circle in a point with negative second coordinate, it is easy to see that $q_1$ decreases as $w_1$ decreases (staying larger than 1) or $w_2$ increases (cf.~Figure~\ref{fig:stab_q}): If $w$ is replaced by a point $w'$ with $1 < w'_1 \leq w_1$ and $w'_2 \geq w_2$, then clearly $\langle p, w' \rangle \leq \langle p, w \rangle = 1 < \langle (1,0), w' \rangle$. Thus, the corresponding tangent for $w'$ meets $\bd(\B^2)$ in a point on the minor circular arc between $p$ and $(1,0)$ and, consequently, the $x$-axis in a point $q'$ with $q'_1 \leq q_1$.

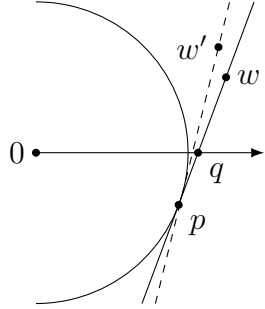
\begin{figure}[ht]
\def\wx{1.25}
\def\wy{0.5}
\def\wdx{1.125}
\def\wdy{0.7}
\def\px{((\wx+\wy*sqrt(\wx^2+\wy^2-1))/(\wx^2+\wy^2))}
\def\py{((\wy-\wx*sqrt(\wx^2+\wy^2-1))/(\wx^2+\wy^2))}
\def\pdx{((\wdx+\wdy*sqrt(\wdx^2+\wdy^2-1))/(\wdx^2+\wdy^2))}
\def\pdy{((\wdy-\wdx*sqrt(\wdx^2+\wdy^2-1))/(\wdx^2+\wdy^2))}
\def\scale{1.85}
\centering
\begin{tikzpicture}[scale=\scale]
\draw[->] (0,0) -- (1.5,0);

\draw (0,-1) arc (-90:90:1);

\draw ({-\px/(\py-1)},-1) -- ({\px/(\py+1)},1);

\draw[dashed] ({-\pdx/(\pdy-1)},-1) -- ({\pdx/(\pdy+1)},1);

\fill (\wx,\wy) circle(\ds) node[anchor=west] {$w$};

\fill (\wdx,\wdy) circle(\ds) node[anchor=east] {$w'$};

\fill ({\px},{\py}) circle(\ds) node[anchor=north west] {$p$};

\fill ({1/\px},0) circle(\ds) node[anchor=north west] {$q$};

\fill (0,0) circle(\ds) node[anchor=east] {$0$};
\end{tikzpicture}
\caption{
An example for the monotonicity of $q_1$ in the proof of Theorem~\ref{thm:stability}:
For $w_1 > 1$ and $w_2 \geq 0$, there is a tangent from $w$ to $\B^2$ that meets $\bd(\B^2)$ in the point $p$ lying in the fourth quadrant. The corresponding tangent for $w'$ with $1 < w'_1 \leq w_1$ and $w'_2 \geq w_2$ meets the $x$-axis in a point to the left of $q$.
}
\label{fig:stab_q}
\end{figure}

In summary, we obtain in any case from \eqref{eq:x(r)} and \eqref{eq:y(r)} that
$$ \left( \frac{x(r)^2 + y(r)^2}{x(r) + y(r) \sqrt{x(r)^2 + y(r)^2 - 1}}, 0 \right) \in K.$$
Since this point is the convex combination of some point in $\B^2$ and $(x(r),y(r))$, where $x(r) > 1$, it follows
$$ s(r) := \frac{x(r)^2 + y(r)^2}{x(r) + y(r) \sqrt{x(r)^2 + y(r)^2 - 1}} \leq x(r) < r \sqrt{2}. $$
Since $v = (0,r \sqrt{2}) \in K$, we conclude that $s(r) \CC^2 \subseteq K$. Combining this with the inclusion \eqref{eq:stab_inclusion1}, we get
$$ d_{BM}(K,\CP^2) \leq \frac{\sqrt{2}}{r} \cdot \frac{1}{s(r)} = \frac{\sqrt{2}}{r} \cdot \frac{x(r) + y(r) \sqrt{ x(r)^2 + y(r)^2 - 1 }}{x(r)^2 + y(r)^2}. $$
The desired estimate now follows directly from Lemma~\ref{lem:stab_poly_ineq}.
\end{proof}

Theorem~\ref{thm:stability} states that if a planar symmetric convex body is far from the Euclidean disc in the Banach-Mazur distance, then it is close to the parallelogram. Consequently, at least one these two distances is not too large for any convex body. We formalize this observation in the corollary below.

\begin{cor}
\label{cor:BM_cover}
For any symmetric convex body $K \subseteq \R^2$ we have $d_{BM}(K, \B^2) < d$ or $d_{BM}(K, \CP^2) < d$, where $d := \frac{11 \sqrt{2}}{10 + \sqrt{2}} < 1.363$.
\end{cor}
\noindent\emph{Proof.}
Let $K \subseteq \R^2$ be a symmetric convex body and suppose that $d_{BM}(K,\B^2) \geq d$. Then Theorem~\ref{thm:stability} yields 
$$ d_{BM}(K,\CP^2) < 1 + \frac{10}{\sqrt{2}} \left( \sqrt{2} - d \right) = d.\eqno\qed $$

A straightforward estimate in the other direction comes from the unit ball of the $\ell_4^2$ space, for which it is well-known that it is at distance $\sqrt[4]{2} \approx 1.189$ to both $\B^2$ and $\CP^2$ (see for example \cite[Proposition~$37.6$]{tomczak}). As there is a quite large gap between $1.189$ and $1.363$, it is not clear how close the above estimate is to the optimal one. It is interesting to note that the problem about covering the symmetric Banach-Mazur compactum with balls centered at $\B^2$ and $\CP^2$ has been proposed during the open problem session of the workshop ``Interplay between Geometric Analysis and Discrete Geometry'' that was held in $2023$ in Mexico (see \cite{workshopreport} for the report available online). In general, it is natural to study the possibility of covering the Banach-Mazur compactum with some $k \geq 1$ balls, with centers either in some specific convex bodies or completely arbitrary ones. For large $k$ and arbitrary centers, this question was studied by Bronstein in \cite{bronstein}. However, for small $k$, apart from the well-studied case of $k=1$ and our planar result for $k=2$, it seems that no other results in this direction are known. It should be noted that Stromquist~\cite{stromquist} constructed a planar symmetric convex body that is of distance at most $\sqrt{1.5} \approx 1.225$ to any other symmetric convex body.

\begin{remark}
A closer analysis of the left-hand side in Lemma~\ref{lem:stab_poly_ineq} suggests a likely improvement of the results obtained in this section, albeit at the cost of additional technicalities. Numerical data suggests that Lemma~\ref{lem:stab_poly_ineq} remains valid if the right-hand side is replaced by $1 + 6.64 \sqrt{2} (1 - r)$, where $6.64 \sqrt{2} \approx 9.39$. Consequently, this would lead to an improvement of the constant $c$ in Theorem~\ref{thm:stability} to $6.64$ and the upper bound on at least one of the Banach-Mazur distances in Corollary~\ref{cor:BM_cover} to $1.36$. However, such constants would probably still be not the best possible (we recall that for a linear constant in the stability estimate, an obvious lower bound is $\frac{1}{\sqrt{2}}$).
\end{remark}

Taking into account the results of this and the previous section, it is natural to ask about the stability of the parallelotope and the cross-polytope as the only three-dimensional symmetric convex bodies with the maximal distance to $\B^3$. However, this situation is more complicated to handle since there is not just a single maximizer. For example, this makes it hard to obtain an analog of Lemma~\ref{lem:orth} for dimension $3$. Instead, a possible way forward might be to use Ader decompositions. The proof of Corollary~\ref{cor:estimate} already shows for a symmetric convex body $K \subseteq \R^3$ with $\B^3 \subseteq K \subseteq d_{BM}(K,\B^3) \B^3$ that the matrix underlying the Ader decomposition for $K$ must be ``close'' to being a multiple of the identity matrix.

\section{Maximal Distance Between Planar \texorpdfstring{$1$}{1}-Symmetric Convex Bodies}
\label{sec:1sym}

In this section, we determine the maximal distance between planar $1$-symmetric convex bodies and characterize the equality case. The main idea to obtain the estimate is to use the $1$-symmetry assumption to place the convex bodies involved in a similar position. This is also already enough to show that the equality case requires one of the convex bodies to be a square. However, to obtain the full characterization of the equality case, we have to consider other linear transformations as well. We begin with the following simple, but somewhat technical, lemma.

\begin{lem}
\label{lem:stab_ineq}
Let $\frac{1}{2} \leq x \leq \frac{1}{\sqrt{2}}$ and $\frac{1}{\sqrt{2}} \leq y \leq 1$ be reals with $x \leq \frac{y(\sqrt{2}-1)}{1-y(2-\sqrt{2})}$. Then
$$\frac{\sqrt{2}y(2x-1)(1-x)}{x + y - 2xy} \leq \sqrt{2} x + y - \sqrt{2},$$
with equality if and only if $(x,y) = \left( \frac{1}{2}, \frac{1}{\sqrt{2}} \right)$ or $(x,y) = \left( \frac{1}{\sqrt{2}},1 \right)$.
\end{lem}

\begin{proof}
We have
$$ x + y -2xy = (x-y)^2 + x (1-x) + y (1-y) > 0. $$
Thus, the desired inequality can be equivalently rewritten as
$$\sqrt{2}y(2x-1)(1-x) \leq (\sqrt{2} x + y - \sqrt{2})(x+y-2xy)$$
or
\begin{equation}
\label{eq:ineq_xy}
\sqrt{2}x^2 + (-2y^2 + y - \sqrt{2})x + y^2 \geq 0.
\end{equation}
The left-hand side is for fixed $y \in \R$ a quadratic function in $x$, which attains its global minimum at $x^*(y) = \frac{2y^2-y+\sqrt{2}}{2\sqrt{2}}$. We claim now that
\begin{equation}
\label{eq:ineq_xy2}
\min \left\{ \frac{1}{\sqrt{2}}, \frac{y(\sqrt{2}-1)}{1-y(2-\sqrt{2})} \right\} \leq \frac{2y^2-y+\sqrt{2}}{2\sqrt{2}}.
\end{equation}
The inequality 
$$ \frac{1}{\sqrt{2}} \leq \frac{2y^2-y+\sqrt{2}}{2\sqrt{2}} $$
holds when
$$y \geq \frac{1}{4} + \frac{1}{4} \sqrt{17 - 8 \sqrt{2}} \approx 0.846.$$
In particular, it is satisfied for $y \geq 0.85$. Moreover, since $1-y(2-\sqrt{2}) \geq 1 - (2-\sqrt{2}) > 0$, the inequality
$$ \frac{y(\sqrt{2}-1)}{1-y(2-\sqrt{2})} \leq \frac{2y^2-y+\sqrt{2}}{2\sqrt{2}} $$
can be equivalently rewritten as
$$ p(y) := (4-2\sqrt{2}) y^3 +  (\sqrt{2}-4) y^2 + 3y - \sqrt{2} \leq 0. $$
It is easy to check that the derivative of this third degree polynomial $p$ is positive everywhere, so $p$ is increasing. Thus, it is enough to calculate directly that $p(0.85)$ is negative (it is approximately equal to $-0.013$). Hence, the inequality \eqref{eq:ineq_xy2} follows.

Now, to prove \eqref{eq:ineq_xy}, it is by \eqref{eq:ineq_xy2} and
$$x \leq \min \left\{ \frac{1}{\sqrt{2}}, \frac{y(\sqrt{2}-1)}{1-y(2-\sqrt{2})} \right\}$$
enough to verify it at $x = \frac{1}{\sqrt{2}}$ and $x=\frac{y(\sqrt{2}-1)}{1-y(2-\sqrt{2})}$, as the considered quadratic function is strictly decreasing in $x\in (-\infty,x^*(y)]$. For $x = \frac{1}{\sqrt{2}}$, the inequality \eqref{eq:ineq_xy} in $y$ can be rewritten as
$$(\sqrt{2}-1)(1-y)\left ( y - \frac{1}{\sqrt{2}} \right ) \geq 0,$$
which is true by the assumptions. For $x = \frac{y(\sqrt{2}-1)}{1-y(2-\sqrt{2})}$, the inequality \eqref{eq:ineq_xy} can be rewritten as
$$ 2 (\sqrt{2}-1)y(y-1)^2 \left( y-\frac{1}{\sqrt{2}} \right) \geq 0, $$
which is also true.

Finally, we see that equality holds if and only if $x \in \{ \frac{1}{\sqrt{2}}, \frac{y(\sqrt{2}-1)}{1-y(2-\sqrt{2})} \}$ and $y \in \{ \frac{1}{\sqrt{2}}, 1 \}$. For the resulting four pairs of values, it is straightforward to verify that only $(x,y) = ( \frac{1}{2}, \frac{1}{\sqrt{2}} )$ and $(x,y) = ( \frac{1}{\sqrt{2}},1 )$ satisfy all assumptions in the lemma, so the conclusion follows.
\end{proof}

The proof of Theorem~\ref{thm:1sym} is split into two lemmas. The following lemma establishes the inequality and part of the equality case. The full characterization of the equality case is covered by a second lemma below.

\begin{lem}
\label{lem:1symm_ineq}
Let $K, L \subseteq \R^2$ be $1$-symmetric convex bodies. Then
$$ d_{BM}(K, L) \leq \sqrt{2}, $$
and if equality holds, then one of $K$ and $L$ is a square.
\end{lem}
\begin{proof}
Since $K$ is $1$-symmetric, we may, by applying a rotation by $45^\circ$ if necessary, assume that
$$\|(\pm 1, 0)\|_K = \|(0, \pm 1)\|_K \leq \frac{1}{\sqrt{2}} \|(1, \pm 1)\|_K = \frac{1}{\sqrt{2}} \|(-1, \pm 1)\|_K$$
and similarly for $L$. Additionally, by applying appropriate dilatations, we can further assume that $\|(1,0)\|_K = 1$ and $L \subseteq K$ with $\bd(L) \cap \bd(K) \neq \emptyset$. Our goal is to prove under these assumptions that $K \subseteq \sqrt{2} L$, which then yields the desired inequality. Because of the $1$-symmetry of $K$ and $L$, it is enough to prove for $K' = \{ x \in K: \ x_1 \geq x_2 \geq 0\}$ and $L'=\{ x \in L: \ x_1 \geq x_2 \geq 0\}$ that $K' \subseteq \sqrt{2}L'$.

From the above assumptions it follows that $v^1 := (1,0), v^2 := (\rho,\rho) \in \bd(K)$ for some $\rho \in (0,\frac{1}{\sqrt{2}}]$ and $w^1 := (\eta,0), w^2 := (\tau,\tau) \in \bd(L)$ for some $\eta \in (0,1]$ and $\tau \in (0,\min \{ \frac{\eta}{\sqrt{2}},\rho \}]$. Since $K$ and $L$ are $1$-symmetric, lines perpendicular to $(1,0)$ support $K$ and $L$ at $v^1$ and $w^1$, respectively, and similarly lines perpendicular to $(1,1)$ support $K$ and $L$ at $v^2$ and $w^2$, respectively. The two respective lines that support $K$ intersect in the point $v := (1, 2\rho - 1)$, and the ones supporting $L$ intersect in $w := (\eta, 2\tau - \eta)$. Thus, $\rho \geq \frac{1}{2}$ and $\tau \geq \frac{\eta}{2}$. Moreover,
$$ \conv \{ 0, v^1, v^2 \} \subseteq K' \subseteq \conv \{ 0, v^1, v^2, v \} $$
and
$$ \conv\{ 0, w^1, w^2 \} \subseteq L' \subseteq \conv \{ 0, w^1, w^2, w \} $$
(cf.~Figure~\ref{fig:1symm_general_situation}).
\newpage

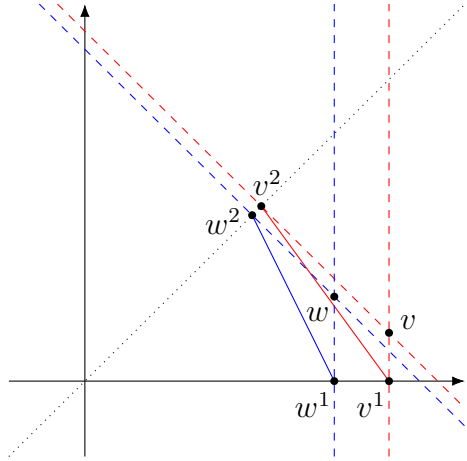
\begin{figure}[ht]
\def\r{0.58}
\def\e{0.82}
\def\t{0.55}
\def\scale{4}
\centering
\begin{tikzpicture}[scale=\scale]
\draw[->] (-0.25,0) -- (1.25,0);
\draw[->] (0,-0.25) -- (0,1.25);
\draw[dotted] (-0.25,-0.25) -- (1.25,1.25);

\draw[red] (1,0) -- (\r,\r);
\draw[red,dashed] (1,-0.25) -- (1,1.25);
\draw[red,dashed] (2*\r-1.25,1.25) -- (1.25,2*\r-1.25);

\draw[blue] (\e,0) -- (\t,\t);
\draw[blue,dashed] (\e,-0.25) -- (\e,1.25);
\draw[blue,dashed] (2*\t-1.25,1.25) -- (1.25,2*\t-1.25);

\fill (1,0) circle(\ds);
\node[anchor=north east] at (1.025,0) {$v^1$};
\fill (\r,\r) circle(\ds);
\node[anchor=south] at (\r+0.02,\r) {$v^2$};
\fill (1,2*\r-1) circle(\ds);
\node[anchor=south west] at (1,2*\r-1.025) {$v$};

\fill (\e,0) circle(\ds);
\node[anchor=north east] at (\e+0.025,0) {$w^1$};
\fill (\t,\t) circle(\ds);
\node[anchor=east] at (\t,\t-0.02) {$w^2$};
\fill (\e,2*\t-\e) circle(\ds);
\node[anchor=north east] at (\e+0.02,2*\t-\e) {$w$};
\end{tikzpicture}
\caption{
An example of the situation in the proof of Lemma~\ref{lem:1symm_ineq}: The blue (resp.~red) solid segment belongs to $L$ (resp.~$K$), whereas the blue (resp.~red) dashed lines support $L$ (resp.~$K$). Since $\bd(K) \cap \bd(L) \neq \emptyset$, $w$ cannot lie in the interior of $K$ and must therefore lie to the top right of the red segment.
}
\label{fig:1symm_general_situation}
\end{figure}

To prove the inclusion $K' \subseteq \sqrt{2}L'$, it is enough to show that $v \in \sqrt{2}L$, as then by $1$-symmetry of $\sqrt{2}L$ we would also have $v^1, v^2 \in \sqrt{2}L$ and in consequence
$$K' \subseteq \conv\{ 0, v^1, v^2, v \} \subseteq \sqrt{2}L.$$
It is thus sufficient to verify that $v \in \conv\{ 0, \sqrt{2}w^1, \sqrt{2}w^2 \}$. In other words, we want to write $v$ in the form $\lambda_1w^1 + \lambda_2w^2$ with $\lambda_1, \lambda_2 \geq 0$ and $\lambda_1+\lambda_2 \leq \sqrt{2}$. Clearly, if 
$$v=(1, 2\rho-1)=\lambda_1w^1+\lambda_2w^2 = (\lambda_1 \eta + \lambda_2 \tau, \lambda_2 \tau),$$
then $\lambda_1=\frac{2-2\rho}{\eta}$ and $\lambda_2 = \frac{2\rho - 1}{\tau}$. Such $\lambda_1, \lambda_2$ are always non-negative by $\rho \in [\frac{1}{2},\frac{1}{\sqrt{2}}]$, and the condition $\lambda_1 + \lambda_2 \leq \sqrt{2}$ can be rewritten as
\begin{equation}
\label{eq:ineq_lambda}
\frac{(2\rho-1)\eta}{\tau} \leq \sqrt{2}\eta + 2 \rho - 2.
\end{equation}

To establish \eqref{eq:ineq_lambda}, we shall use the fact that $\bd(L) \cap \bd(K) \neq \emptyset$. In particular, the $1$-symmetry of $K$ and $L$ shows $w \not \in \inte(\conv\{ 0, v^1, v^2 \})$. This means that if we write $w$ as $w=\mu_1v^1 + \mu_2v^2$ with $\mu_1, \mu_2 \geq 0$, then $\mu_1 + \mu_2 \geq 1$. It is easy to calculate that $\mu_1 = 2 \eta - 2 \tau \geq 0$ and $\mu_2 = \frac{2\tau - \eta}{\rho} \geq 0$, so
\begin{equation}
\label{eq:ineq_tau}
\tau \geq \frac{\rho + \eta - 2\rho \eta}{2(1-\rho)}.
\end{equation}
Since $\rho \geq \frac{1}{2}$, $\eta \leq 1$, and $\tau \leq \frac{\eta}{\sqrt{2}}$ this yields
$$ \frac{1}{2} = \frac{\rho - (2 \rho - 1)}{2 (1 - \rho)} \leq \frac{\rho - \eta (2 \rho - 1)}{2 (1 - \rho)} = \frac{\rho + \eta - 2\rho \eta}{2(1-\rho)} \leq \frac{\eta}{\sqrt{2}}$$
and in particular $\eta \geq \frac{1}{\sqrt{2}}$. Furthermore, the last inequality can be rewritten as
$$ \rho \leq \frac{\eta (\sqrt{2} - 1)}{1 - \eta (2 - \sqrt{2})}. $$
By \eqref{eq:ineq_tau}, we have now
\begin{equation}
\label{eq:ineq_eq}
\frac{(2\rho-1)\eta}{\tau} \leq \frac{2\eta(2\rho-1)(1-\rho)}{\rho + \eta - 2\rho \eta}
\end{equation}
(we note that $\rho + \eta - 2\rho \eta=(\rho-\eta)^2 + \rho(1-\rho) + \eta(1-\eta)>0$). To establish \eqref{eq:ineq_lambda}, it is in summary enough to prove
$$\frac{\sqrt{2}\eta(2\rho-1)(1-\rho)}{\rho + \eta - 2\rho \eta} \leq \eta + \sqrt{2} \rho - \sqrt{2}.$$
This follows directly from Lemma~\ref{lem:stab_ineq}, as all the assumptions are met. This concludes the proof of the inclusion $K \subseteq \sqrt{2}L$.

Assume now that the equality $d_{BM}(K,L) = \sqrt{2}$ holds. We proved before that $v$ belongs to the triangle with vertices $\{0, \sqrt{2}w^1, \sqrt{2}w^2\}$. If $d_{BM}(K,L) = \sqrt{2}$ holds, then $v$ clearly has to be on the side $[\sqrt{2}w^1, \sqrt{2}w^2]$, as otherwise $v$ would be in the interior of $L$ and thus also $K \subseteq \inte(\sqrt{2}L)$ by $1$-symmetry of $K$ and $L$. In particular, we need to have equality in the estimate \eqref{eq:ineq_lambda}. Looking at the estimate \eqref{eq:ineq_eq}, we see that we need to have equality in the final estimate following from Lemma~\ref{lem:stab_ineq}. Therefore, by the equality condition in Lemma~\ref{lem:stab_ineq}, we have $(\rho, \eta)=\left ( \frac{1}{2}, \frac{1}{\sqrt{2}} \right )$ or $(\rho, \eta)=\left ( \frac{1}{\sqrt{2}}, 1 \right )$. In the first case, we have $v^2 = \frac{1}{2} (1,1) \in \bd(K)$. Since a line perpendicular to $(1,1)$ supports $K$ at $v^2$ and $v^1 = (1,0) \in K$, it follows that
$$ K'= \conv \left\{ 0, (1,0), \frac{1}{2} (1,1) \right\}. $$
The $1$-symmetry of $K$ now shows $K = \CC^2$, so $K$ is a square in this case and the conclusion follows.

In the second case, the equality in \eqref{eq:ineq_eq} implies $\tau=\frac{1}{2}$. Therefore, we have $w^1 = (1,0) \in L$ and $w^2 = \frac{1}{2} (1,1) \in \bd(L)$. The conclusion follows now in the same way as before, this time showing that $L=\CC^2$ is a square. This finishes the proof. 
\end{proof}

Let us point out another direct consequence of the above proof for the equality case.  From the conditions on $\rho$, $\eta$, and $\tau$ in the equality case it follows that if $L$ is a square then $K$ must satisfy
$$ \|(1,0)\|_K = \frac{1}{\sqrt{2}} \|(1,1)\|_K. $$
In other words, rotating $K$ by $45^\circ$ does not change the $K$-norm of the points on the coordinate axes and their angle bisectors. However, this condition is not enough to guarantee that a $1$-symmetric convex body $K$ has the distance $\sqrt{2}$ to the square. In the following lemma, we describe the equality condition fully.

\begin{lem}
\label{lem:1symm_equal_cond}
Let $K \subseteq \R^2$ be a $1$-symmetric convex body and let $\varphi \colon \R^2 \to \R^2$ be a rotation by $45^\circ$. Then $d_{BM}(K,\CP^2) = \sqrt{2}$ holds if and only if
\begin{equation}
\label{eq:1symm_equal_cond}
\| x\|_K \| x\|_{\varphi(K^\circ)} \geq \| x \|^2
\end{equation}
for every $x \in \R^2$.
\end{lem}

We again note that the rotation direction of $\varphi$ does not matter here. planar $1$-symmetric convex bodies are invariant under rotation by $90^\circ$, so rotating $K^\circ$ by $45^\circ$ in either direction yields the same result.

\begin{proof}
First, assume $d_{BM}(K,\CP^2) = \sqrt{2}$. By homogeneity, it suffices to prove the required inequality for arbitrary $x \in \bd(K)$. To this end, we define the square
$$ P_x := \conv\{ x, \varphi^2(x), \varphi^4(x), \varphi^6(x) \}. $$
Since $K$ is $1$-symmetric, it is invariant under rotation by $90^\circ$, i.e., under $\varphi^2$. Therefore, we have $P_x \subseteq K$. Now, $d_{BM}(K,\CP^2) = \sqrt{2}$ implies $K \not\subseteq \inte(\sqrt{2} P_x)$, so by convexity of $K$, there must exist some $y \in \bd(\sqrt{2} P_x) \cap K$. Using the $1$-symmetry of $K$ again, we may assume $y \in \sqrt{2} [x,\varphi^6(x)]$ and observe that
\begin{equation}
\label{eq:P_x_inner_rep}
\langle \varphi^{-1}(x), \varphi^6(x) \rangle = \langle x, \varphi^{-1}(x) \rangle = \|x\| \|\varphi^{-1}(x)\| \cos(45^\circ) = \frac{\|x\|^2}{\sqrt{2}} =: \beta_x.
\end{equation}
Thus, by using the fact that $y$ lies on the segment $\sqrt{2} [x,\varphi^6(x)]$, we get
$$ \|x\|_{\varphi(K^\circ)} = \|\varphi^{-1}(x)\|_{K^\circ} \geq \langle y, \varphi^{-1}(x) \rangle = \beta_x \sqrt{2} = \|x\|^2. $$
Hence, the required inequality is proved.

Now, assume that $K$ satisfies \eqref{eq:1symm_equal_cond}. We observe for $x \in \bd(K)$ that $\inte(P_x) \subseteq \inte(K)$. Moreover, \eqref{eq:P_x_inner_rep} shows
$$ P_x = \left\{ y \in \R^2 : |\langle \varphi^k(x), y \rangle| \leq \beta_x, \: k=1,3 \right \}, $$
so for $k \in \{0, 2, 4, 6\}$ there is no intersection point of $K$ with the set
$$ \varphi^k(x) + (0,\infty) (\varphi^k(x) - \inte(P_x)) =\left\{ y \in \R^2 : \langle \varphi^{k-1}(x),y \rangle > \beta_x \: \text{ and } \: \langle \varphi^{k+1}(x), y \rangle > \beta_x \right\}, $$
as this would contradict $\varphi^k(x) \in \bd(K)$ (cf.~Figure~\ref{fig:1symm_equal_cond}). Hence, any point $p \in K$ satisfies
\begin{equation}
\label{eq:P_x_cond}
|\langle \varphi(x), p \rangle| \leq \beta_x
\quad \text{or} \quad
|\langle \varphi^3(x), p \rangle| \leq \beta_x.
\end{equation}

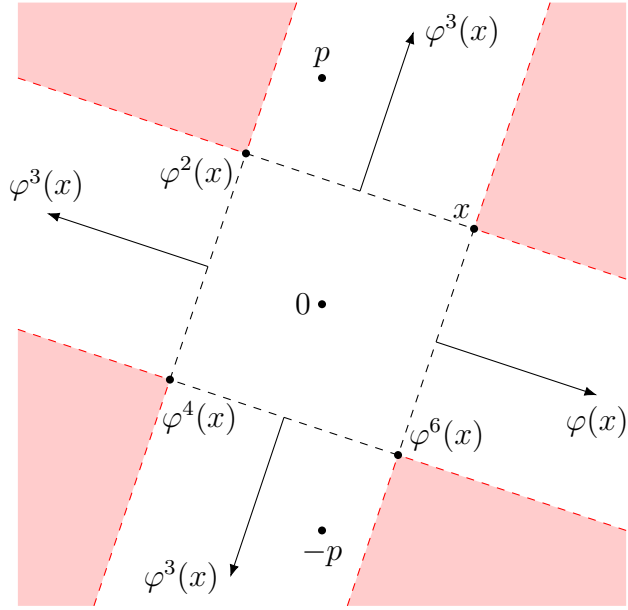
\begin{figure}[ht]
\def\scale{0.9}
\centering
\begin{tikzpicture}[scale=\scale]
\fill[red!20] (3,4) -- (2,1) -- (4,1/3) -- (4,4) -- cycle;
\fill[red!20] (-4,3) -- (-1,2) -- (-1/3,4) -- (-4,4) -- cycle;
\fill[red!20] (-3,-4) -- (-2,-1) -- (-4,-1/3) -- (-4,-4) -- cycle;
\fill[red!20] (4,-3) -- (1,-2) -- (1/3,-4) -- (4,-4) -- cycle;

\draw[dashed,red] (3,4) -- (2,1);
\draw[dashed] (2,1) -- (1,-2);
\draw[dashed,red] (1/3,-4) -- (1,-2);

\draw[dashed,red] (-4,3) -- (-1,2);
\draw[dashed] (-1,2) -- (2,1);
\draw[dashed,red] (4,1/3) -- (2,1);

\draw[dashed,red] (-3,-4) -- (-2,-1);
\draw[dashed] (-2,-1) -- (-1,2);
\draw[dashed,red] (-1/3,4) -- (-1,2);

\draw[dashed,red] (4,-3) -- (1,-2);
\draw[dashed] (1,-2) -- (-2,-1);
\draw[dashed,red] (-4,-1/3) -- (-2,-1);

\draw[->] (3/2,-1/2) -- ({3/2+3/sqrt(2)},{-1/2-1/sqrt(2)});
\node[anchor=north] at ({3/2+3/sqrt(2)-0.25},{-1/2-1/sqrt(2)}) {$-\varphi^3(x)$};
\draw[->] (1/2,3/2) -- ({1/2+1/sqrt(2)},{3/2+3/sqrt(2)});
\node[anchor=west] at ({1/2+1/sqrt(2)},{3/2+3/sqrt(2)}) {$\varphi(x)$};
\draw[->] (-3/2,1/2) -- ({-3/2-3/sqrt(2)},{1/2+1/sqrt(2)});
\node[anchor=south] at ({-3/2-3/sqrt(2)},{1/2+1/sqrt(2)}) {$\varphi^3(x)$};
\draw[->] (-1/2,-3/2) -- ({-1/2-1/sqrt(2)},{-3/2-3/sqrt(2)});
\node[anchor=east] at ({-1/2-1/sqrt(2)},{-3/2-3/sqrt(2)}) {$-\varphi(x)$};

\fill (2,1) circle(\ds);
\node[anchor=south east] at (2.1,1) {$x$};
\fill (-1,2) circle(\ds);
\node[anchor=north east] at (-1,2.1) {$\varphi^2(x)$};
\fill (-2,-1) circle(\ds);
\node[anchor=north west] at (-2.25,-1.15) {$\varphi^4(x)$};
\fill (1,-2) circle(\ds);
\node[anchor=south west] at (1,-2.1) {$\varphi^6(x)$};

\fill (0,3) circle(\ds) node[anchor=south] {$p$};
\fill (0,-3) circle(\ds) node[anchor=north] {$-p$};

\fill (0,0) circle(\ds) node[anchor=east] {$0$};
\end{tikzpicture}
\caption{
An example of the situation in the proof of Lemma~\ref{lem:1symm_equal_cond}: The black dashed square $P_x$ is a subset of $K$. Thus, no point $p \in K$ can lie in one of the open red areas (with red dashed boundaries), as this would contradict $x \in \bd(K)$. Hence, any point $p \in K$ must lie between at least one pair of dashed parallel lines.
}
\label{fig:1symm_equal_cond}
\end{figure}

Now, let $P$ be an origin symmetric parallelogram that satisfies for $d = d_{BM}(K, \CP^2)$ that
$$ P \subseteq K \subseteq d P. $$
Then there exist $x \in \bd(K)$ and $p \in K$ such that $P = \conv\{ \pm x, \pm p \}$. By \eqref{eq:P_x_cond}, there further exists $k \in \{1,3\}$ with $|\langle \varphi^k(x), p \rangle| \leq \beta_x$ and thus $\langle \varphi^k(x), y \rangle \leq \beta_x$ for all $y \in P$. Consequently, we have $\|\varphi^k(x)\|_{P^\circ} \leq \beta_x$. Since $K \subseteq d P$ implies $\| \cdot \|_{K^\circ} \leq d \| \cdot \|_{P^\circ}$, we obtain from the condition \eqref{eq:1symm_equal_cond} and $\|\varphi^{k+1}(x)\|_K = 1$ by $1$-symmetry of $K$ that
$$ \frac{\|x\|^2}{\sqrt{2}} = \beta_x \geq \| \varphi^k(x) \|_{P^\circ} \geq \frac{\| \varphi^k(x) \|_{K^\circ}}{d} = \frac{\| \varphi^{k+1}(x) \|_{\varphi(K^\circ)}}{d} \geq \frac{\| \varphi^{k+1}(x) \|^2}{d} = \frac{\|x\|^2}{d}. $$
This rearranges to $d \geq \sqrt{2}$. Since $d \leq \sqrt{2}$ is also true by Lemma~\ref{lem:1symm_ineq}, we obtain the claimed equality $d=d_{BM}(K,\CP^2) = \sqrt{2}$.
\end{proof}

The two previous lemmas combined directly yield Theorem~\ref{thm:1sym}. The final goal of this section is to investigate condition \eqref{eq:1symm_equal_cond} further. Let us first give a more geometric interpretation of \eqref{eq:1symm_equal_cond} in the remark below.

\begin{remark}
\label{rem:1symm_cond}
We use the notation given in Lemma~\ref{lem:1symm_equal_cond}. Since $\varphi(K^\circ) = (\varphi(K))^\circ$, we can use the support function of $\varphi(K)$ to rewrite condition \eqref{eq:1symm_equal_cond} for $x \in \bd(K)$ as
$$ h_{\varphi(K)}(x) = \max \{ \langle x, y \rangle : y \in \varphi(K) \} \geq \| x \|^2 = \langle x, x \rangle. $$
In other words, if we draw a line perpendicular to $x$ through $x$ itself, then $\varphi(K)$ must contain some point in the halfspace bounded by this line and not containing the origin. This is in turn equivalent to $\varphi(K)$ containing a point $y$ such that $y = \lambda x$ for some $\lambda \in [1,\infty)$ or $\conv\{0,x,y\}$ is a triangle with a right or obtuse angle at $x$.
\end{remark}

With the above observation, it is immediately clear that $K$ being invariant under a rotation by $45^\circ$ is sufficient to imply $\eqref{eq:1symm_equal_cond}$.

\begin{cor}
Let $K \subseteq \R^2$ be a $1$-symmetric convex body that is invariant under rotation by $45^\circ$. Then
$$ d_{BM}(K,\CP^2) = \sqrt{2}. $$
\end{cor}

In particular, the above corollary generalizes a result by Lassak in \cite{lassak}, where it is proved that regular $8j$-gons, $j \in \N$, are at Banach-Mazur distance $\sqrt{2}$ from $\CP^2$. It is now natural to ask if the reverse implication is also true, that is, if condition \eqref{eq:1symm_equal_cond} is for a $1$-symmetric convex body actually equivalent to being invariant under rotation by $45^\circ$. However, Example~\ref{ex:stab_equal_cond_bodies} below shows that this is not the case in general. Moreover, because there is quite a lot of freedom in choosing the convex curve $\gamma$ in the construction provided below, it seems rather unlikely that condition \eqref{eq:1symm_equal_cond} could be expressed in some much simpler way.

\begin{exam}
\label{ex:stab_equal_cond_bodies}
Let $a = (1,0)$, $b = \frac{1}{\sqrt{2}}(1,1)$, and choose some $v \in \frac{a}{2} + \frac{1}{2} \B^2$ that also lies in the interior of the triangle
$$ T := \conv \left\{ a, b, \frac{a + b}{\sqrt{2}} \right\} $$
(cf.~Figure~\ref{fig:non-symmetric_example_1-symm}). Moreover, let $\gamma \colon [0,1] \to T$ be any convex curve with $\gamma(0) = a$, $\gamma(1) = b$, $\gamma(\frac{1}{4}) = v$, such that for any $t \in [0,\frac{1}{4}]$, $\gamma(t) \in \frac{a}{2} + \frac{1}{2} \B^2$ and $\gamma(1-t) \in \frac{b}{2} + \frac{1}{2} \B^2$, and for any $t \in [\frac{1}{4},\frac{1}{2}]$, $\gamma(t)$ is the reflection of $\gamma(1-t)$ at the line $V := \lin\{ a+b \}$. It is clear that there exists a unique planar $1$-symmetric convex body $K \subseteq \R^2$ whose boundary between $a$ and $b$ coincides with the set $\gamma([0, 1])$.

We show that such $K$ satisfies the condition \eqref{eq:1symm_equal_cond}. To do so, it is by Remark~\ref{rem:1symm_cond} enough to show that if $x \in \gamma([0,1])$, then the rotation of $K$ by $45^\circ$ contains a point $y$ such that $y = x$ or $\conv \{0,x,y\} $ is a triangle with a right or obtuse angle at $x$. Note that rotating $K$ by $45^\circ$ or reflecting it at the line $V$ yields by the $1$-symmetry of $K$ the same result, so it suffices to find $z \in K$ such that its reflection $y$ at the line $V$ fulfills one of the above properties.

Now, let $t \in [0,1]$ be such that $x = \gamma(t)$. If $t \in [\frac{1}{4},\frac{3}{4}]$, then $z = \gamma(1-t)$ is by the choice of $\gamma$ the reflection of $x$ at $V$ and thus an eligible choice. If $t \in [0,\frac{1}{4}]$, then $x \in \frac{a}{2} + \frac{1}{2} \B^2$. Thus, $x = a$ or the triangle $\conv \{0,x,a\} $ has a right or obtuse angle at $x$ since $[0,a]$ is a diameter of $\frac{a}{2} + \frac{1}{2} \B^2$. In either case, we may choose $z = b$. Finally, if $t \in [\frac{3}{4},1]$, we can argue analogously with roles of $a$ and $b$ exchanged that $z = a$ is a possible choice.

\begin{figure}[ht]
\def\r{0.25}
\def\ang{22.5}
\def\scale{7.5}
\centering
\begin{tikzpicture}[scale=\scale]
\draw[dashdotted] (0,0) -- (1,{sqrt(2)-1});
\draw[dashdotted] (0,0) -- (1,0);
\draw[dashdotted] (0,0) -- ({1/sqrt(2)},{1/sqrt(2)});

\draw[dotted] (1,0) arc (0:180:0.5);
\draw[dotted] (0,0) arc (-135:45:0.5);

\draw[blue,dashed] (1,0) -- (1,{sqrt(2)-1});
\draw[blue,dashed] ({1/sqrt(2)},{1/sqrt(2)}) -- (1,{sqrt(2)-1});
\draw[blue,dashed] ({1/sqrt(2)},{1/sqrt(2)}) -- ({(1+sqrt(2))/sqrt(8)},{1/sqrt(8)});
\draw[blue,dashed] (1,0) -- ({(1+sqrt(2))/sqrt(8)},{1/sqrt(8)});

\draw[red] (1,0) -- (0.98,0.08) -- (0.9,0.28)
		-- ({(1.8*cos(\ang)+0.56*sin(\ang))*cos(\ang) - 0.9},{(1.8*cos(\ang)+0.56*sin(\ang))*sin(\ang) - 0.28}) -- ({1/sqrt(2)},{1/sqrt(2)});

\fill[red] (1,0) circle(\ds);
\node[anchor=west] at (1,0) {$a$};
\fill (0.5,0) circle(\ds) node[anchor=south] {$\frac{a}{2}$};

\fill[red] ({1/sqrt(2)},{1/sqrt(2)}) circle(\ds);
\node[anchor=south west] at ({1/sqrt(2)},{1/sqrt(2)}) {$b$};
\fill ({1/sqrt(8)},{1/sqrt(8)}) circle(\ds) node[anchor=north west] {$\frac{b}{2}$};

\fill (1,{sqrt(2)-1}) circle(\ds) node[anchor=south west] {$\frac{a+b}{\sqrt{2}}$};

\fill[red] (0.9,0.28) circle(\ds);
\node[anchor=south west] at (0.9,0.28) {$v$};
\fill[red] ({(1.8*cos(\ang)+0.56*sin(\ang))*cos(\ang) - 0.9},
            {(1.8*cos(\ang)+0.56*sin(\ang))*sin(\ang) - 0.28}) circle(\ds);
\fill[red] (0.98,0.08) circle(\ds);

\fill (0,0) circle(\ds) node[anchor=east] {$0$};
\end{tikzpicture}
\caption{
An example for the convex curve $\gamma$ (red) in Example~\ref{ex:stab_equal_cond_bodies}.
}
\label{fig:non-symmetric_example_1-symm}
\end{figure}
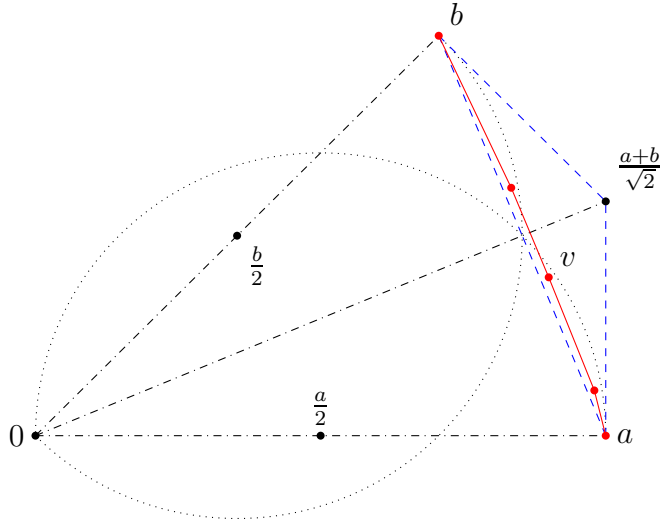

As seen in Figure~\ref{fig:non-symmetric_example_1-symm}, $K$ does not necessarily need to be invariant under reflection at the line $V$ (or equivalently under rotation by $45^\circ$) as $\gamma(t)$ does not need to be the reflection at $V$ of any point in $\gamma([0,1])$ for $t<\frac{1}{4}$.
\end{exam}
\newpage

We conclude the paper with the observation that the inequality
$$d_{BM}(K,L) \leq \sqrt{n}$$
is not true for all $1$-symmetric convex bodies $K,L \subseteq \R^n$ and arbitrary dimension $n$. It is shown in \cite{kobosvarivoda} that in the case $n = 3$ we have $d_{BM}(\CP^3, \CC^3) = \frac{9}{5} > \sqrt{3}$. Since the convex body $\CP^n$ (or $\CC^n$) is involved in all maximal distance situations that are considered throughout the present paper, it is perhaps natural to conjecture that $\frac{9}{5}$, which is the distance between $\CP^3$ and $\CC^3$, could be the maximal possible distance between two $1$-symmetric convex bodies in $\R^3$. However, we currently do not have much evidence to support such a conjecture.

\bigskip

\textsc{Department of Mathematics, Technical University of Munich, Germany} \\
\textit{E-mail address}: \textbf{florian.grundbacher@tum.de}.

\vskip 0.2in

\textsc{Faculty of Mathematics and Computer Science, Jagiellonian University in Cracow, Poland} \\
\textit{E-mail address}: \textbf{tomasz.kobos@uj.edu.pl}

\vfill\eject

\end{document}